\documentclass{amsart}
\usepackage{amsmath, amsfonts}
\usepackage{amsthm}
\usepackage{amssymb}

\newtheorem{theorem}{Theorem}
\newtheorem{lemma}{Lemma}
\newtheorem{claim}{Claim}[section]

\newtheorem{case}{Case}
\newtheorem{casenew}{Case}
\newtheorem{corollary}{Corollary}
\newtheorem{conjecture}{Conjecture}

\theoremstyle{definition}

\theoremstyle{remark}

\numberwithin{equation}{section}

\begin{document}

\title{Ramsey numbers of trees and unicyclic graphs versus fans}

\author{Matthew Brennan}
\address{Massachusetts Institute of Technology, Cambridge, MA 02139}
\email{brennanm@mit.edu}


\maketitle

\begin{abstract}
The generalized Ramsey number $R(H, K)$ is the smallest positive integer $n$ such that for any graph $G$ with $n$ vertices either $G$ contains $H$ as a subgraph or its complement $\overline{G}$ contains $K$ as a subgraph. Let $T_n$ be a tree with $n$ vertices and $F_m$ be a fan with $2m + 1$ vertices consisting of $m$ triangles sharing a common vertex. We prove a conjecture of Zhang, Broersma and Chen for $m \ge 9$ that $R(T_n, F_m) = 2n - 1$ for all $n \ge m^2 - m + 1$. Zhang, Broersma and Chen showed that $R(S_n, F_m) \ge 2n$ for $n \le m^2 -m$ where $S_n$ is a star on $n$ vertices, implying that the lower bound we show is in some sense tight. We also extend this result to unicyclic graphs $UC_n$, which are connected graphs with $n$ vertices and a single cycle. We prove that $R(UC_n, F_m) = 2n - 1$ for all $n \ge m^2 - m + 1$ where $m \ge 18$. In proving this conjecture and extension, we present several methods for embedding trees in graphs, which may be of independent interest.
\end{abstract}

\section{Introduction}

Given two graphs $H$ and $K$, the generalized Ramsey number $R(H, K)$ is the smallest positive integer $n$ such that for any graph $G$ with $n$ vertices, either $G$ contains $H$ as a subgraph or the complement $\overline{G}$ of $G$ contains $K$ as a subgraph. When both $H$ and $K$ are complete graphs, $R(H, K)$ is the classical Ramsey number. Because classical Ramsey numbers are difficult to determine, Chv\'{a}tal and Harary proposed to study generalized Ramsey numbers of graphs other than complete graphs through a series of papers in 1972 and 1973 \cite{chvatal1972generalized1, chvatal1972generalized2, chvatal1973generalized}. 

Generalized Ramsey numbers have since been well studied for a variety of graphs, including trees and fans. Chv\'{a}tal determined the Ramsey number of trees versus complete graphs, showing that $R(T_n, K_m) = (n - 1)(m - 1) + 1$ for positive integers $m$ and $n$ \cite{chvatal1977tree}. Burr, Erd\H{o}s, Faudree, Rousseau and Schelp determined the Ramsey number of large trees versus odd cycles, showing that $R(T_n, C_m) = 2n - 1$ for odd $m \ge 3$ and $n \ge 756 m^{10}$ \cite{burr1982ramsey}. Recently, we showed this result is also true for smaller trees satisfying $n \ge 25m$ \cite{brennan2016}. Salman and Broersma determined the Ramsey number of paths versus fans, finding $R(P_n, F_m)$ for various ranges of $n$ and $m$ \cite{salman2006path}. Shi determined the Ramsey number of cycles versus fans, showing that $R(C_n, F_m) = 2n - 1$ for all $n > 3m$ \cite{shi2010ramsey}. In \cite{li1996fan}, Li and Rousseau proved an upper bound on the Ramsey number of fans versus complete graphs, showing
$$R(F_m, K_n) \le (1 + o(1)) \frac{n^2}{\log n}.$$
A survey of Ramsey numbers and related lower bounds can be found in \cite{radziszowski1994small}.

There have also been general lower bounds shown to hold for Ramsey numbers. In 1981, Burr proved the following lower bound in terms of the chromatic number $\chi(G)$ of a graph $G$ and its chromatic surplus $s(G)$ -- the minimum number of vertices in a color class over all proper vertex colorings of $G$ using $\chi(G)$ colors.

\begin{theorem}[Burr \cite{burr1981ramsey}]
If $H$ is a connected graph with $n$ vertices and $s(K)$ is the chromatic surplus of the graph $K$, then for $n \ge s(K)$ we have
$$R(H,K) \ge (n-1)(\chi(K) - 1) + s(K).$$
\end{theorem}

The Ramsey numbers of trees versus odd cycles, of cycles versus fans and of trees versus complete graphs determined by Burr \textit{et al.}, Shi and Chv\'{a}tal, respectively, achieve Burr's lower bound. Note that for a fan or odd cycle $K$, $\chi(K) = 3$ and $s(K) = 1$ and thus Theorem 1 implies that $R(T_n, F_m) \ge 2n - 1$ for all $m$ and $n \ge s(K) = 1$. This lower bound can be seen directly by considering the complete bipartite graph $K_{n-1, n-1}$. Since $K_{n-1, n-1}$ is triangle-free, it does not contain $F_m$ as a subgraph. Furthermore, $\overline{K_{n-1, n-1}}$ consists of two connected components of size $n - 1$ and thus does not contain $T_n$ as a subgraph.

In 2015, Zhang, Broersma and Chen showed that the lower bound from Theorem 1 is tight for large trees and stars versus fans, proving the following two theorems. Here, $S_n$ denotes a star on $n$ vertices consisting of an independent set of $n - 1$ vertices all adjacent to a single vertex.

\begin{theorem}[Zhang, Broersma, Chen \cite{zhang2015ramsey}]
$R(T_n, F_m) = 2n-1$ for all integers $m$ and $n \ge 3m^2 - 2m - 1$.
\end{theorem}

\begin{theorem}[Zhang, Broersma, Chen \cite{zhang2015ramsey}]
$R(S_n, F_m) = 2n-1$ for all integers $n \ge m^2 - m + 1$ and $m \neq 3, 4, 5$, and this lower bound is the best possible. Moreover, $R(S_n, F_m) = 2n - 1$ for $n \ge 6m - 6$ and $m = 3, 4, 5$.
\end{theorem}

Because it is generally believed that $R(T_n, G) \le R(S_n, G)$ for any graph $G$, Zhang, Broersma and Chen made the following conjecture based on Theorem 3.

\begin{conjecture}[Zhang, Broersma, Chen \cite{zhang2015ramsey}]
$R(T_n, F_m) = 2n - 1$ for all integers $m \ge 6$ and $n \ge m^2 - m + 1$.
\end{conjecture}

Theorem 3 yields that if $n \le m^2 - m$ then $R(S_n, F_m) \ge 2n$, implying $n \ge m^2 - m + 1$ is the best achievable lower bound on $n$ in terms of $m$ over which $R(T_n, F_m) = 2n - 1$ is true \cite{zhang2015ramsey}. In this paper, we prove Conjecture 1 for the case $m \ge 9$. Specifically, we prove the following theorem.

\begin{theorem}
$R(T_n, F_m) = 2n - 1$ for all $n \ge m^2 - m + 1$ for $m \ge 9$.
\end{theorem}

In \cite{zhang2015ramsey}, Zhang, Broersma and Chen also determined $R(T_n, K_{\ell - 1} + mK_2)$ as a corollary of Theorem 2. Here, $mG$ denotes the union of $m$ vertex-disjoint copies of $G$ and $G_1 + G_2$ is the graph obtained by joining every vertex of $G_1$ to every vertex of $G_2$ in $G_1 \cup G_2$. Zhang, Broersma and Chen identify $R(T_n, K_{\ell - 1} + mK_2)$ for $n \ge 3m^2 - 2m - 1$ by induction on $\ell$, using Theorem 2 as a base case. Their induction argument remains valid when Theorem 4 is used as the base case, yielding the following updated version of their corollary.

\begin{corollary}[Zhang, Broersma, Chen \cite{zhang2015ramsey}]
$R(T_n, K_{\ell - 1} + mK_2) = \ell(n - 1) + 1$ for $\ell \ge 2$ and $n \ge m^2 - m + 1$ where $m \ge 9$.
\end{corollary}

We also extend Theorem 4 from trees to unicyclic graphs. Let $UC_n$ denote a particular connected graph with $n$ vertices and a single cycle -- or equivalently a connected graph with $n$ vertices and $n$ edges. We prove the following result.

\begin{theorem}
$R(UC_n, F_m) = 2n - 1$ for all $n \ge m^2 - m + 1$ for $m \ge 18$.
\end{theorem}

Note that Theorem 5 implies Theorem 4 as a corollary in the case $m \ge 18$. Despite this, we present our proofs of these two theorems separately because our approach to Theorem 4 motivates our proof of Theorem 5 and because we require a sufficiently different approach and more careful analysis to prove Theorem 4 for $9 \le m < 18$. The next section provides the notation and key lemmas that will be used in the proofs of Theorem 4 and Theorem 5. In the two subsequent sections, we prove Theorem 4 and Theorem 5.

\section{Preliminaries and Lemmas}

We first provide the notation we will adopt on proving Theorem 4 and Theorem 5. Let $G$ be any simple graph. Here, $d_X(v)$ denotes the degree of a vertex $v$ in the set $X \subseteq V(G)$ in $G$ and $\overline{d_X}(v)$ denotes the degree of $v$ in $X$ in the complement graph $\overline{G}$. Similarly, $N_X(v)$ and $\overline{N_{X}}(v)$ denote the sets of neighbors of $v$ in the set $X$ in $G$ and $\overline{G}$, respectively. It is clear that $d_X(v) + \overline{d_X}(v) = |X|$ for any $X \subseteq V(G)$ not containing $v$ and that $d_X(v) = |N_X(v)|$ and $\overline{d_{X}}(v) = |\overline{N_X}(v)|$. We also extend this notation to $N_X(Y)$ and $\overline{N_X}(Y)$ for sets $Y \subseteq V(G)$ disjoint from $X$. When the set $X$ is omitted, it is implicitly $V(G)$ where the graph $G$ is either clear from context or explicitly stated. We denote the maximum and minimum degrees of a graph $G$ as $\Delta(G)$ and $\delta(G)$, respectively. When $G$ is bipartite and connected, we let the sets $A(G)$ and $B(G)$ denote the partite sets of $G$ with $V(G) = A(G) \cup B(G)$ and $|A(G)| \ge |B(G)|$. In particular, this implies that $|A(G)| \ge |V(G)|/2$. For a tree $T$, we let $L(T)$ denote the set of leaves of $T$. Also note that if $T$ is a tree then since $T$ is bipartite, $A(T)$ and $B(T)$ are well-defined.

We now prove two lemmas that will be used throughout the proofs of Theorem 4 and Theorem 5. The first is a structural lemma concerning the vertices of degree two in trees and will be crucial to our methods for embedding trees.

\begin{lemma}
Given a tree $T$ and a subset $F \subseteq V(T)$, there is a set $D$ satisfying:
\begin{enumerate}
\item $D \subseteq A(T)$ and $F \cap D = \emptyset$;
\item each $v \in D$ satisfies $d_T(v) = 2$;
\item each $v \in D$ is not adjacent to any leaves of $T$;
\item no two vertices in $D$ have a common neighbor;
\end{enumerate}
and the size of $D$ is at least
$$|D| \ge \frac{1}{4} \left( |V(T)| - 8 |L(T)| - 2|F| + 12 \right).$$
\end{lemma}

\begin{proof}
Since $|E(T)| = |V(T)| - 1$, it follows that
$$2|L(T)| + 2\left( |V(T)|-|L(T)| \right) -2 = 2|E(T)| = \sum_{v \in V(T)} d_T(v) = |L(T)| + \sum_{v \in I} d_T(v)$$
where $I \subseteq V(T)$ is the set of internal vertices of $T$. Let $H \subseteq V(T)$ be the set of vertices $v \in V(T)$ with $d_T(v) \ge 3$. Rearranging yields that
$$2 + |H| \le 2 + \sum_{v \in H} (d_T(v) - 2) = |L(T)|.$$
Therefore $|H| \le |L(T)| - 2$ and it follows that
$$\sum_{v \in H} d_T(v) = 2|H| + |L(T)| - 2 \le 3 |L(T)| - 6.$$
Therefore $H$ has at most $3|L(T)| - 6$ neighbors in $T$. Now note that for any $v \in V(T)$, either $A(T)$ contains $v$ and none of the neighbors of $v$, or does not contain $v$. Since $d_T(v) \ge 1$ for all $v \in H$ and $v \in L(T)$, this observation implies the following two inequalities
\begin{align*}
|A(T) \cap (N_T(H) \cup H)| &\le \sum_{v \in H} d_T(v) \le 3|L(T)| - 6 \\
|A(T) \cap (N_T(L(T)) \cup L(T))| &\le |L(T)|.
\end{align*}
Let $J \subseteq A(T)$ be the set of $v \in A(T)$ such that $d_T(v) = 2$, $v \not \in F$ and $v$ is not adjacent to any vertex in $H$ or $L(T)$. Since $|A(T)| \ge |V(T)|/2$, it follows that
\begin{align*}
|J| &\ge |A(T)| - |F| - |A(T) \cap (N_T(H) \cup H)| - |A(T) \cap (N_T(L(T)) \cup L(T))|\\
&\ge \frac{1}{2} (|V(T)| - 8 |L(T)| - 2|F| + 12).
\end{align*}

Now consider the graph $K$ on the vertex set $V(K) = J$ such that for distinct $u, v \in J$, the edge $uv \in E(K)$ if and only if $u$ and $v$ have a common neighbor in $T$. Note that $uv \not \in E(T)$ for all $u,v \in J$ since $J \subseteq A(T)$. Since no vertex of $J$ is adjacent to a vertex of $H$, if $uv \in E(K)$ then the common neighbor of $u$ and $v$ in $T$ is not adjacent to any vertices in $J$ other than $u$ and $v$. Therefore there cannot be a cycle in $K$ since otherwise there would be a cycle in $T$, which is not possible since $T$ is a tree. It follows that $K$ is a forest and is bipartite. Let $D$ be the larger part $A(K)$ in a bipartition of $K$. Defined in this way, $D$ satisfies
$$|D| \ge \frac{1}{2} |J| \ge \frac{1}{4} \left( |V(T)| - 8 |L(T)| - 2|F| + 12 \right).$$
Also note that each $v \in D$ satisfies $d_T(v) = 2$, $v \not \in F$ and $v$ is not adjacent to a leaf of $T$ since $D \subseteq J$ and no two vertices in $D$ have a common neighbor in $T$ since $D$ is an independent set in $K$.
\end{proof}

In our proof of Theorem 4, we apply Lemma 1 always with $F = \emptyset$. In our proof of Theorem 5, $F$ consists of two vertices that are adjacent along the cycle in the unicyclic graph $UC_n$. The next lemma asserts that there is always a vertex in any given tree that, when removed, leaves two disconnected sets of similar sizes.

\begin{lemma}
Given a tree $T$ with $|V(T)| \ge 3$, there is a vertex $v \in V(T)$ such that the vertices of the forest $T - v$ can be partitioned into two sets $K$ and $H$ such that there are no edges between $K$ and $H$ and
$$\frac{1}{3}(|V(T)|-1) \le |K|, \, |H| \le \frac{2}{3} (|V(T)|-1).$$
\end{lemma}

\begin{proof}
First note that for any vertex $v \in V(T)$, the forest $T - v$ has $d_T(v)$ connected components. Consider the following procedure. Begin by setting $v$ to be an arbitrary leaf of $T$. If $T - v$ has a connected component $C$ of size at least $|V(T)|/2$, set $v$ to its unique neighbor $u$ in $C$. Note that $T - u$ has a connected component of size $|V(T)| - |C|$ and one or more connected components with the sum of their sizes equal to $|C| - 1$. Therefore either $|C| = |V(T)|/2$ or the size of the largest connected component decreases on setting $v$ to $u$. Thus the procedure either leads to a vertex $v$ such that $T - v$ has a connected component $C$ of size $|C| = |V(T)|/2$ or terminates at a vertex $v$ such that all connected components of $T - v$ have size strictly less than $|V(T)|/2$.

If $v$ is such that $T - v$ has a connected component $C$ of size $|C| = |V(T)|/2$, then let $K = C$ and $H = V(T - v) - C$. It follows that $|H| = |V(T)|/2 - 1$ and $|K| = |C| = |V(T)|/2$, which implies the desired result since in this case $|V(T)|$ must be even and thus $|V(T)| \ge 4$. Note that $d_T(v) = 1$ is not possible by the condition on $v$ and $|V(T)| \ge 3$. If $d_T(v) = 2$, then $T - v$ consists of two connected components $K$ and $H$ such that $|K| + |H| = |V(T)| - 1$. Since $|K|, |H| < |V(T)|/2$, we have that $|K| = |H| = (|V(T)|-1)/2$, in which case the result also holds.

If $d_T(v) = d \ge 3$, let the connected components of $T - v$ be $C_1, C_2, \dots, C_d$. Note that in this case, we must have that $|V(T)| \ge 4$. Assume without loss of generality that $|C_1| \le |C_2| \le \cdots \le |C_d| < |V(T)|/2$. Note that since $d \ge 3$ and $|C_1| + |C_2| + \cdots + |C_d| = |V(T)| - 1$, it follows that $|C_1| \le (|V(T)| - 1)/3$. Now let $t$ be the largest positive integer such that $|C_1| + |C_2| + \cdots + |C_t| \le (|V(T)|-1)/3$. Note that if $t = d - 1$, then $|C_d| \ge 2(|V(T)|-1)/3$ which is not possible since $|C_d| < |V(T)|/2$. Therefore $t \le d - 2$. Note that $|C_1| + |C_2| + \cdots + |C_{t+1}| > (|V(T)|-1)/3$ by definition. If $|C_1| + |C_2| + \cdots + |C_{t+1}| \le 2(|V(T)|-1)/3$, then letting
$$K = C_1 \cup C_2 \cup \cdots \cup C_{t+1} \quad \text{and} \quad H = C_{t+2} \cup C_{t+3} \cup \cdots \cup C_d$$
yields valid sets $K$ and $H$. If $|C_1| + |C_2| + \cdots + |C_{t+1}| > 2(|V(T)|-1)/3$, then it follows that $(|V(T)|-1)/3 < |C_{t+1}| < |V(T)|/2 \le 2(|V(T)|-1)/3$ and letting
$$K = C_{t+1} \quad \text{and} \quad H = C_1 \cup \cdots \cup C_t \cup C_{t+2} \cup \cdots \cup C_d$$
yields the desired sets $K$ and $H$, completing the proof of the lemma.
\end{proof}

The last two lemmas are stated without proof. Lemma 3 is a folklore lemma and Lemma 4 was proven by Zhang, Broersma and Chen and is Lemma 3 in \cite{zhang2015ramsey}.

\begin{lemma}[Folklore]
Let $T$ be a tree with $w_1 \in V(T)$ and let $H$ be a graph with $\delta(H) \ge |V(T)| - 1$ and $w_2 \in V(H)$. Then $T$ can be embedded in $H$ such that $w_1$ is mapped to $w_2$.
\end{lemma}

\begin{lemma}[Zhang, Broersma, Chen \cite{zhang2015ramsey}]
$R(T_n, mK_2) = n + m - 1$ for $n \ge 4m - 4$ where $mK_2$ denotes a matching on $m$ edges.
\end{lemma}

\section{Proof of Theorem 4}

Let $T_n$ denote a particular tree with $n$ vertices. Assume for the sake of contradiction that there is a graph $G$ with $2n - 1$ vertices that does not contain $F_m$ as a subgraph and such that its complement $\overline{G}$ does not contain $T_n$ as a subgraph. Our proof of Theorem 4 begins with a setup similar to that used by Zhang, Broersma and Chen in  \cite{zhang2015ramsey} to prove Theorem 2. In particular, Zhang, Broersma and Chen introduced the sets $X_v, Y_v$ and $U_v$, used a weaker statement of Claim 3.1 and proved a stronger variant of Claim 3.4 that is not necessary for our approach. The remaining claims and final steps in our proof of Theorem 4 consist of several new methods to embed $T_n$ in $\overline{G}$ that allow for a strengthened analysis and yield the desired quadratic lower bound on $n$ in terms of $m$.

Our proof of Theorem 4 uses the following key ideas. First we prove Theorem 4 for all trees $T_n$ containing some vertex adjacent to a large number of leaves. In the general case, we show that $G$ must have enough structure in common with the extremal graph $K_{n-1, n-1}$ to yield a contradiction as follows. The lack of a fan $F_m$ in $G$ implies that the neighborhood of any vertex cannot contain an $m$-matching, which imposes a strong restriction on the structure of neighborhoods in $G$ and guarantees a large independent subset in any neighborhood. Attempting to greedily embed $T_n$ in $\overline{G}$ twice yields that there are two large disjoint independent sets $U_1$ and $U_2$, which we prove together induce a large nearly-complete bipartite subgraph of $G$. Additional methods of embedding $T_n$ in $\overline{G}$ guarantee that most of the edges are present between $U_1$ and $U_2$. Using the fact that $G$ does not contain $F_m$, we have that each vertex in $W = V(G) - (U_1 \cup U_2)$ is adjacent to a small number of vertices in at least one of $U_1$ or $U_2$. From here we divide into the two cases in which $T_n$ has many or few leaves, both of which lead to a contradiction. The case with many leaves is more easily handled and the case with few leaves is handled by examining the vertices of degree two in $T_n$ using Lemma 1. We now begin our proof of Theorem 4.

Consider an arbitrary vertex $v \in V(G)$. Let $M_v \subseteq E(G)$ denote a maximum matching in $N(v)$ and let $t = |M_v|$. If $t \ge m$, then $G$ contains an $F_m$ on the vertices $V(M_v) \cup \{ v \}$ and thus $t \le m - 1$. Let $U_v = N(v) - V(M_v)$. Note that no two vertices in $U_v$ can be adjacent since this would allow $M_v$ to be extended, contradicting its maximality. Now label the vertices in $V(M_v)$ as $x_1, x_2, \dots, x_t, y_1, y_2, \dots, y_t$ where $M_v = \{x_1 y_1, x_2 y_2, \dots, x_t y_t \}$ and $d_{U_v}(x_i) \le d_{U_v}(y_i)$ for all $1 \le i \le t$. If $d_{U_v}(y_i) \ge 2$ and $d_{U_v}(x_i) \ge 1$, then there are two distinct vertices $z, w \in U_v$ such that $x_i z$ and $y_i w$ are edges and $M_v \cup \{x_i z, y_i w \} - \{x_i y_i \}$ is a matching larger than $M_v$, which is a contradiction. Similarly, if $d_{U_v}(y_i) = d_{U_v}(x_i) = 1$, then $x_i$ and $y_i$ must be adjacent to the same vertex in $U_v$ since otherwise $M_v$ can again be extended. In summary, either $d_{U_v}(x_i) = d_{U_v}(y_i) = 1$ and $x_i$ and $y_i$ are adjacent to the same vertex in $U_v$ or $d_{U_v}(x_i) = 0$.

Now assume without loss of generality that $d_{U_v}(y_i) \le 1$ for $1 \le i \le k$ and $d_{U_v}(y_i) \ge 2$ for $k + 1 \le i \le t$ where $0 \le k \le t$. Note that there are at most $k$ vertices in $U_v$ adjacent to vertices in $V(M_v) - \{y_{k+1}, \dots, y_t \}$. Let $Y_v$ be the union of $V(M_v) - \{y_{k+1}, \dots, y_t \}$ and its set of neighbors in $U_v$ and let $X_v = U_v - (U_v \cap Y_v)$. Now note that as defined above, we have
\begin{align}
|U_v| &= d(v) - 2t \ge d(v) - 2m + 2, \\
|X_v| &\ge d(v) - 2t - k \ge d(v) - 3m + 3, \\
|Y_v| &\le t + 2k \le 3m - 3, \\
|X_v| + |Y_v| &= d(v) - t + k \ge d(v) - m + 1.
\end{align}
Note that if $X_v, Y_v$ and $U_v$ are replaced with $X_v \cap S$, $Y_v \cap S$ and $U_v \cap S$ for any set $S \subseteq V(G)$ and $d(v)$ is replaced with $d_S(v)$, then the lower bounds in terms of $m$ shown above still hold by the same argument.

Let $T$ be the largest subgraph of $\overline{G}$ that is an embedding of a subtree of $T_n$ into $\overline{G}$. Since $\overline{G}$ does not contain $T_n$, it follows that $|V(T)| \le n - 1$ and that there is some $v \in V(T)$ such that $v$ has no neighbors permitting $T$ to be extended in $\overline{G}$. This requires that $v$ is adjacent to all vertices in $V(G) - V(T)$. Since $|V(G) - V(T)| \ge 2n - 1 - (n - 1) = n$, it follows that $d(v) \ge n$. This implies by the inequalities above that $|X_v \cup Y_v| \ge d(v) - m + 1 \ge n - m + 1$, $|U_v| \ge n - 2m + 2$ and $|X_v| \ge n - 3m + 3$ by inequalities (3.1)--(3.4) above.

If $|X_v| \ge n - m + 1$, let $X_1$ be a subset of $X_v$ of size $n - m + 1$, let $Y_1 = \emptyset$ and let $U_1 = X_1 \subseteq X_v \subseteq U_v$. If $|X_v| < n - m + 1$ and $|U_v| \ge n - m + 1$, let $X_1 = X_v$, let $U_1$ be a subset of $U_v$ containing $X_1$ such that $|U_1| = n -m + 1$, and let $Y_1 = U_1 - X_1$. If $|U_v| < n - m +1$, let $X_1 = X_v$, let $U_1 = U_v$ and let $Y_1$ be a subset of $Y_v$ containing $U_v \cap Y_v$ such that $|X_1 \cup Y_1| = n - m + 1$. This ensures that $|X_1 \cup Y_1| = n - m + 1$, $|X_1| \ge n - 3m + 3$, $|U_1| \ge n - 2m + 2$, $|Y_1| \le 3m - 3$ and $X_1 \subseteq X_v$, $Y_1 \subseteq Y_v$, $U_1 \subseteq U_v$, $X_1 \subseteq U_1$, $U_1 \subseteq X_1 \cup Y_1$ and $X_1 \cap Y_1 = \emptyset$. Note that since $U_1$ is an independent set, any embedding of a sub-forest of $T_n$ of size at most $|U_1|$ to the subgraph of $\overline{G}$ induced by $U_1$ succeeds. Also note that each $v \in X_1$ is not adjacent to any vertex in $X_1 \cup Y_1$ other than itself.

The remainder of the proof of the theorem is divided into several claims. The first claim proves Theorem 4 for a class of trees. The proof of this claim is adapted from the beginning of the proof of Theorem 5 in \cite{zhang2015ramsey}.

\begin{claim}
Each vertex $v \in V(T_n)$ is adjacent to at most $m - 2$ leaves of $T_n$ and $\Delta(T_n) < 11n/20$.
\end{claim}

\begin{proof}
We begin by showing that $\Delta(G) \le n + m - 2$. Assume for contradiction that $d(u) \ge n + m - 1$ for some $u \in V(G)$. Since $m \ge 4$, it follows that $n \ge m^2 - m + 1 > 4m - 4$ and by Lemma 4, either the subgraph of $\overline{G}$ induced by $N(u)$ contains $T_n$ or the subgraph of $G$ induced by $N(u)$ contains $mK_2$, which along with $u$ yields $F_m$ as a subgraph of $G$. This is a contradiction and thus $\Delta(G) \le n + m - 2$.

Now we show that each vertex of $T_n$ is adjacent to at most $m - 2$ leaves. Assume for the sake of contradiction that some $v \in V(T_n)$ is adjacent to at least $m - 1$ leaves. Let $L$ be a set of $|L| = m - 1$ leaves adjacent to $v$. Note that $T_n - L$ is a tree satisfying that
$$\delta(\overline{G}) \ge (2n - 1) - 1 - \Delta(G) \ge n - m = |V(T_n - L)| - 1.$$
By Theorem 3, it follows that $\overline{G}$ contains $S_n$ as a subgraph. Let $w$ be the center of this $S_n$ subgraph. By Lemma 3, $T_n - L$ can be embedded in $\overline{G}$ such that $v$ is mapped to $w$. Since $w$ is the center of an $S_n$ subgraph of $\overline{G}$, it follows that $\overline{d}(w) \ge n - 1$ and therefore including $|L| = m - 1$ neighbors of $w$ disjoint from the embedding of $T_n - L$ yields a successful embedding of $T_n$ into $\overline{G}$. This is a contradiction and thus each vertex of $T_n$ is adjacent to at most $m - 2$ leaves.

We now show that $\Delta(T_n) < 11n/20$. Assume for the sake of contradiction that $\Delta(T_n) \ge 11n/20$ and let $v$ be a vertex of largest degree in $T_n$. Let $K \subseteq L(T_n)$ be the set of leaves adjacent to $v$ and let $H \subseteq V(T_n)$ be the set of vertices of degree at least two adjacent to $v$. It follows that $d_{T_n}(v) = |K| + |H|$. Because $T_n$ contains $v$, each leaf in $K$, each vertex in $H$ and at least $|H|$ distinct additional vertices each adjacent to a vertex of $H$, it follows that $n = |V(T_n)| \ge 1 + |K| + 2|H|$. Therefore
$$|K| + n \ge 1 + 2|K| + 2|H| = 1 + 2 d_{T_n}(v) \ge 1 + 11n/10$$
which implies that $|K| \ge 1 + n/10 > m - 1$ since $n \ge m^2 - m + 1 > 10m - 20$ for $m \ge 9$. Applying the previous argument now yields a contradiction.
\end{proof}

The next claim guarantees embeddings of sub-forests of $T_n$ to $X_1 \cup Y_1$, allowing some flexibility over where a set of vertices is mapped as long as not too many vertices in the set are in $A(T_n)$.

\begin{claim}
Let $H \subseteq V(T_n)$ be such that $|V(T_n - H)| \le |X_1| + |Y_1| = n - m + 1$ or, equivalently, such that $|H| \ge m - 1$. Let $w_1, w_2, \dots, w_k \in V(T_n - H)$ be distinct vertices of $T_n - H$ and let $u_1, u_2, \dots, u_k \in X_1$ be distinct vertices of $X_1$. If
$$|A(T_n) \cap \{ w_1, w_2, \dots, w_k \}| + |A(T_n) \cap H| \le m + 1$$
then $T_n - H$ can be embedded in $\overline{G}$ such that the vertices of $T_n - H$ are mapped to a subset of $X_1 \cup Y_1$ and $w_i$ is mapped to $u_i$ for all $1 \le i \le k$.
\end{claim}

\begin{proof}
Suppose that $|V(T_n - H)| = |X_1| + |Y_1| = n - m + 1$. Note that
\begin{align*}
|A(T_n) \cap V(T_n - H)| &= |A(T_n)| - |A(T_n) \cap H| \\
&\ge n/2 + |A(T_n) \cap \{ w_1, w_2, \dots, w_k \}| - m - 1 \\
&\ge 3m - 3 + |A(T_n) \cap \{ w_1, w_2, \dots, w_k \}| \\
&\ge |Y_1| + |A(T_n) \cap \{ w_1, w_2, \dots, w_k \}|
\end{align*}
since $n \ge m^2 - m + 1 \ge 8m - 4$ for all $m \ge 9$. Now embed the forest $T_n - H$ to $\overline{G}$ as follows. Map $w_i$ to $u_i$ for all $1 \le i \le k$ and map $|Y_1|$ vertices in $A(T_n) \cap V(T_n - H)$, excluding $w_i$ if $w_i \in A(T_n)$ for all $1 \le i \le k$, to $Y_1$ arbitrarily, which is possible by the inequality above. Map the remaining $|X_1| - k$ vertices of $T_n - H$ arbitrarily to distinct vertices in $X_1 - \{ u_1, u_2, \dots, u_k \}$. Since each vertex in $X_1$ is adjacent to all vertices in $X_1 \cup Y_1$ in $\overline{G}$ other than itself and $A(T_n) \cap V(T_n - H)$ is an independent set of $T_n$, this embedding succeeds. If $|V(T_n - H)| < |X_1| + |Y_1|$, then the claim is also implied by this argument.
\end{proof}

The next claim begins the main body of the proof of Theorem 4. It is the first claim towards showing that $G$ contains a large induced bipartite subgraph that is nearly complete.

\begin{claim}
There is a vertex $u \in V(G)$ such that $d_{O_1}(u) \ge n$ where $O_1 = V(G) - (X_1 \cup Y_1)$.
\end{claim}

\begin{proof}
First note that $T_n$ has a connected subtree $K$ on $|X_1| + |Y_1| = n - m + 1$ vertices. Such a subtree can be obtained by removing a leaf of $T_n$ and repeatedly removing a leaf of the resulting tree until $m - 1$ vertices have been removed. Now note that $|A(T_n) \cap V(T_n - K)| \le |V(T_n - K)| = m - 1 < m + 1$. Thus by Claim 3.2 applied with $k = 0$, this subtree $K$ can be embedded in $\overline{G}$ such that the vertices $V(K)$ are mapped exactly to $X_1 \cup Y_1$.

Now consider the following procedure to greedily extend this embedding to an embedding of $T_n$ in $\overline{G}$. At any point in this procedure, let $H$ denote the subgraph of $\overline{G}$ that has so far been embedded to. Note that $V(H) = X_1 \cup Y_1$ initially and $X_1 \cup Y_1 \subseteq V(H)$ at any point in this procedure. Greedily extend the embedding as follows: if $u_1 \in V(T_n)$ has been mapped to $u_2 \in V(H) \subseteq V(G)$, $w \in V(T_n)$ has not been embedded to $\overline{G}$ and $w$ is adjacent to $u_1$ in $T_n$, map $w$ to some element of $\overline{N_{V(G)-V(H)}}(u_2)$ if it is non-empty. Note that each possible $w$ has a unique neighbor among the vertices $V(T_n)$ that have been embedded to $H$ since $T_n$ contains no cycles and $H$ is always connected throughout this procedure. Furthermore, if not all of $T_n$ has been embedded to $\overline{G}$ then some such $w \in V(T_n)$ must exist since $T_n$ is connected. Since $\overline{G}$ does not contain $T_n$ as a subgraph, the embedding must fail with $\overline{N_{V(G)-V(H)}}(u_2) = \emptyset$ for some $u_2 \in V(H)$ where $|V(H)| \le n - 1$. Since $V(G) - V(H) \subseteq O_1$, this implies that $d_{O_1}(u_2) \ge d_{V(G) - V(H)}(u_2) \ge |V(G) - V(H)| \ge n$, completing the proof of the claim.
\end{proof}

We now construct sets $X_2, Y_2$ and $U_2$ similarly to $X_1, Y_1$ and $U_1$. Since $d_{O_1}(u) \ge n$, it follows that  $|O_1 \cap (X_u \cup Y_u)| \ge d_{O_1}(u) - m + 1 \ge n - m + 1$ and $|O_1 \cap U_u| \ge n - 2m + 2$ by applying the lower bounds (3.1)--(3.4) restricted to the subset $S = O_1$. If $|O_1 \cap X_u| \ge n - m + 1$, let $X_2$ be a subset of $O_1 \cap X_u$ of size $n - m + 1$, let $Y_2 = \emptyset$ and let $U_2 = X_2 \subseteq O_1 \cap X_u \subseteq O_1 \cap U_u$. If $|O_1 \cap X_u| < n - m + 1$ and $|O_1 \cap U_u| \ge n - m + 1$, let $X_2 = O_1 \cap X_u$, let $U_2$ be a subset of $O_1 \cap U_u$ containing $X_2$ such that $|U_2| = n -m + 1$ and let $Y_2 = U_2 - X_2$. If $|O_1 \cap U_u| < n - m +1$, let $X_2 = O_1 \cap X_u$, let $U_2 = O_1 \cap U_u$ and let $Y_2$ be a subset of $O_1 \cap Y_u$ containing $O_1 \cap U_u \cap Y_u$ such that $|X_2 \cup Y_2| = n - m + 1$. Applying the lower bounds (3.1)--(3.4) restricted to the subset $S = O_1$ yields that $|X_2 \cup Y_2| = n - m + 1$, $|X_2| \ge n - 3m + 3$, $|U_2| \ge n - 2m + 2$ and $|Y_2| \le 3m - 3$. Furthermore, we have that $X_2 \subseteq O_1 \cap X_u$, $Y_2 \subseteq O_1 \cap Y_u$, $U_2 \subseteq O_1 \cap U_u$, $X_2 \subseteq U_2$ and $U_2 \subseteq X_2 \cup Y_2$. Also note that $X_1, X_2, Y_1$ and $Y_2$ are pairwise disjoint and that Claim 3.2 holds when $X_1$ and $Y_1$ are replaced with $X_2$ and $Y_2$, respectively.

We now proceed to the fourth claim in our proof of Theorem 4, which shows that each vertex in $U_1$ is adjacent to almost all vertices in $U_2$, and thus that $U_1$ and $U_2$ induce a nearly complete bipartite subgraph of $G$.

\begin{claim}
For each $w \in U_1$, we have $\overline{d_{U_2}}(w) \le 2m - 3$, and for each $w \in U_2$, we have $\overline{d_{U_1}}(w) \le 2m - 3$.
\end{claim}

\begin{proof}
Assume for contradiction that for some $w \in U_1$, it holds that $\overline{d_{U_2}}(w) \ge 2m - 2$. By Lemma 2, there is a vertex $x$ of $T_n$ such that $V(T_n - x)$ can be partitioned into two sets $K$ and $H$ such that there are no edges between $K$ and $H$ and $(n-1)/3 \le |K|, |H| \le 2(n-1)/3$. Note that $K$ and $H$ are both sub-forests of $T_n$. Let the connected components of $K$ be $C_1, C_2, \dots, C_d$ where $|C_1| \le |C_2| \le \cdots \le |C_d|$. Let $p$ be the minimum positive integer such that $|C_1| + |C_2| + \cdots + |C_{p}| \ge 2m - 2$ and let $C = C_1 \cup C_2 \cup \cdots \cup C_p$. Note that $p$ exists since $|C_1| + |C_2| + \cdots + |C_d| = |K| \ge (n - 1)/3 \ge (m^2 - m)/3 \ge 2m - 2$ for $m \ge 6$. Since $|C_i| \ge 1$ for all $1 \le i \le d$, it follows that $p \le 2m - 2$ and since $C \subseteq K$ it follows that $|C| \le |K| \le 2(n-1)/3$.

We have that $T_n - C$ is a tree since $C$ is a union of connected components of $T_n - x$ and $x \in V(T_n - C)$. Also note that $|V(T_n - C)| = n - |C| \le n - 2m + 2 \le |U_1|$. Now consider the following embedding of $T_n$ into $\overline{G}$. Map $x$ to $w$ and the remaining $|V(T_n - C)| - 1 \le |U_1| - 1$ vertices of $T_n - C$ arbitrarily to distinct vertices in $U_1 - \{ w \}$. Now let $v_1, v_2, \dots, v_p$ be the unique vertices adjacent to $x$ in $C_1, C_2, \dots, C_p$, respectively, in $T_n$. Map $v_1, v_2, \dots, v_p$ to $p \le 2m - 2$ distinct neighbors of $w$ in $U_2$ in $\overline{G}$. Map $C - \{ v_1, v_2, \dots, v_p \}$ arbitrarily to distinct vertices in $U_2$ that have not already been mapped to. Note this is possible since $|C| \le 2(n-1)/3 < n - 2m + 2 \le |U_2|$ is implied by $n \ge m^2 - m + 1 > 6m - 8$ for $m \ge 6$. Since $U_1$ and $U_2$ are complete subgraphs of $\overline{G}$, this embedding succeeds and contradicts the fact that $T_n$ is not a subgraph of $\overline{G}$. The same argument as above shows that for each $w \in U_2$, it holds that $\overline{d_{U_1}}(w) \le 2m - 3$.
\end{proof}

Let $Z = V(G) - (X_1 \cup Y_1 \cup X_2 \cup Y_2)$ and note that $|Z| = 2n - 1 - 2(n - m + 1) = 2m - 3$. Also let $W = V(G) - (U_1 \cup U_2)$ and note that $Z \subseteq W$. We now prove the fifth and sixth claims in our proof of Theorem 4. Claim 3.5 is an intermediary step towards showing Claim 3.6, which guarantees that each vertex $w \in W$ has few neighbors in at least one of $U_1$ or $U_2$ and provides a natural way to associate each such $w$ with one of the two sets.

\begin{claim}
For each $w \in W$, either $\overline{d_{U_1}} (w) < 11n/40$ or $\overline{d_{U_2}}(w) < 11n/40$.
\end{claim}

\begin{proof}
Assume for contradiction that some $w \in W$ satisfies that $\overline{d_{U_1}} (w) \ge 11n/40$ and $\overline{d_{U_2}}(w) \ge 11n/40$. By Lemma 2, there is a vertex $x$ of $T_n$ such that $V(T_n - x)$ can be partitioned into two sets $K$ and $H$ such that there are no edges between $K$ and $H$ and $(n-1)/3 \le |K|, |H| \le 2(n-1)/3$.

Now suppose that $d_{K}(x) \le \overline{d_{U_1}} (w)$ and $d_{H}(x) \le \overline{d_{U_2}}(w)$. Consider the following embedding of $T_n$ to $\overline{G}$. Map $x$ to $w$, map $N_{K}(x)$ to $d_K(x)$ distinct vertices in $\overline{N_{U_1}}(w)$ and map $N_H(x)$ to $d_H(x)$ distinct vertices in $\overline{N_{U_2}}(w)$. Map the remaining vertices of $K$ to $|K| - d_K(x)$ distinct vertices in $U_1$ that have not been embedded to and map the remaining vertices of $H$ to $|H| - d_H(x)$ distinct vertices of $U_2$ that have not been embedded to. Note this is possible since $|K|, |H| \le 2(n-1)/3 \le n - 2m + 2 \le |U_1|, |U_2|$ is implied by $n \ge m^2 - m + 1 > 6m - 8$ for $m \ge 6$. This yields a successful embedding of $T_n$ to $\overline{G}$, which is a contradiction.

Therefore either $d_{K}(x) > \overline{d_{U_1}} (w) \ge 11n/40$ or $d_{H}(x) > \overline{d_{U_2}}(w) \ge 11n/40$. Without loss of generality, assume that $d_{K}(x) > \overline{d_{U_1}} (w) \ge 11n/40$. Note that $K$ is the union of $d_K(x)$ connected components of $T_n - x$. Let $C$ be the union of $d_K(x) - \overline{d_{U_1}} (w)$ of these connected components. Let $K' = K - C$ and $H' = H \cup C$. Note that $d_{K'}(x) = \overline{d_{U_1}} (w)$ and that
$$d_{H'}(x) = d_{T_n}(x) - d_{K'}(x) \le \Delta(T_n) - \overline{d_{U_1}} (w) < 11n/40 \le \overline{d_{U_2}}(w)$$
where $\Delta(T_n) < 11n/20$ by Claim 3.1. It follows that $|K'| \ge d_{K'}(x) = \overline{d_{U_1}} (w) \ge 11n/40 \ge 2m - 2$ since $n \ge m^2 - m + 1 \ge 80(m-1)/11$ for $m \ge 8$. This implies that $|H'| = n - 1 - |K'| \le n - 2m + 1 < |U_2|$. Furthermore $|K'| \le |K| \le |U_1|$. Now applying the embedding described in the previous case to $K'$ and $H'$ in place of $K$ and $H$ yields the same contradiction.
\end{proof}

\begin{claim}
For each $w \in W$, either $d_{U_1} (w) \le m - 1$ or $d_{U_2}(w) \le m - 1$.
\end{claim}

\begin{proof}
By Claim 3.5, either $\overline{d_{U_1}} (w) < 11n/40$ or $\overline{d_{U_2}}(w) < 11n/40$ for each $w \in W$. Given some $w \in W$, suppose that $\overline{d_{U_2}} (w) < 11n/40$ and assume for contradiction that $d_{U_1}(w) \ge m$. Let $S = N_{U_2}(w)$ be the set of neighbors of $w$ in $U_2$. Note that since $\overline{d_{U_2}} (w) < 11n/40$, it follows that $|S| = d_{U_2}(w) > |U_2| - 11n/40 \ge 29n/40 - 2m + 2$. By Claim 3.4, each $x \in N_{U_1} (w) \subseteq U_1$ satisfies that $\overline{d_{U_2}}(x) \le 2m - 3$. This implies that for each $x \in N_{U_1} (w)$, we have that $d_S (x) \ge |S| - \overline{d_{U_2}}(x) \ge 29n/40 - 4m + 5 > m$ since $n \ge m^2 -m  + 1 > 200(m - 1)/29$ for $m \ge 7$. This implies that there is a matching with $m$ edges between $N_{U_1}(w)$ and $S$. With the vertex $w$, this matching yields that $G$ contains $F_m$ as a subgraph, which is a contradiction. By the same argument, if $\overline{d_{U_1}} (w) < 11n/40$ then it must follow that $d_{U_2}(w) \le m - 1$.
\end{proof}

In the remainder of the proof, we divide into the two cases in which $T_n$ has many and few leaves, both of which lead to a contradiction. Since $|Z| = 2m - 3$, Claim 3.6 implies that either at least $m - 1$ vertices $w \in Z$ satisfy that $d_{U_1} (w) \le m - 1$ or at least $m - 1$ vertices $w \in Z$ satisfy that $d_{U_2} (w) \le m - 1$. Without loss of generality, assume that there is a set $Z_1$ such that $|Z_1| = m - 1$ and each $w \in Z_1$ satisfies that $d_{U_1} (w) \le m - 1$. Let $Z_1 = \{ z_1, z_2, \dots, z_{m-1} \}$. We now consider the following two cases.

\begin{case}
$|L(T_n)| \ge m + 1$.
\end{case}

We claim there is a set $D$ consisting of $|D| = m - 1$ leaves chosen such that $d \le m - 4$ where $d$ is the maximum value $d_{D}(x)$ over all $x \in V(T_n) - D$. We choose the set $D$ as follows. Let $y \in V(T_n) - L(T_n)$ be such that $d_{L(T_n)}(y)$ is maximized. By Claim 3.1, it follows that $d_{L(T_n)}(y) \le m - 2$. If $d_{L(T_n)}(y) \le m - 4$, then any choice of $m - 1$ leaves forms a valid set $D$. Otherwise, let $D$ consist of $m - 4$ leaves adjacent to $y$ and three leaves not adjacent to $y$. Note that this is possible since $|L(T_n) - N_{L(T_n)}(y)| \ge m + 1 - (m - 2) = 3$. Since $m \ge 7$ and $L(T_n)$ is an independent set, this choice of $D$ ensures that $d \le m - 4$, verifying our claim.

Now let $D = \{d_1, d_2, \dots, d_{m-1} \}$ and let $N_{T_n}(D) = \{ y_1, y_2, \dots, y_k \}$ where $k = |N_{T_n}(D)|$. Note that $d_D(y_i) \le d$ for all $1 \le i \le d$ and there is some $i$ such that $d_D(y_i) = d$. Since each $d_i$ is adjacent to exactly one $y_i$, it follows that $k \le m - 1$. Consider the following embedding of $T_n$ to $\overline{G}$. Begin by mapping $d_i$ to $z_i$ for all $1 \le i \le m - 1$. For each $1 \le i \le k$, let $N_i$ be the subset of $Z_1$ that $N_D(y_i)$ is embedded to, or in other words the set of all $z_j$ such that $d_j \in N_D(y_i)$. Note that $|N_i| \le d$ for each $1 \le i \le k$. By Claim 3.6, each $z_i$ is not adjacent to at most $m - 1$ vertices of $X_1 \subseteq U_1$ in $\overline{G}$. Thus the number of vertices in $X_1$ adjacent to all vertices in $N_i$ in $\overline{G}$ is at least
\begin{align*}
|X_1| - |N_i| \cdot (m - 1) &\ge n - 3m + 3 - d(m - 1) \\
&\ge m^2 - 4m + 4 - d(m - 1) \\
&= (m - 3 - d)(m - 1) + 1 \\
&> m - 1 \ge k
\end{align*}
where the last inequality follows from $d \le m - 4$. Therefore greedily selecting $k$ vertices $u_1, u_2, \dots, u_k$ such that $u_i$ is adjacent to all vertices in $N_i$ in $\overline{G}$ and $u_i \in X_1$ for all $1 \le i \le k$ succeeds. Now note that for each $1 \le i \le k$, either $y_i \in A(T_n)$ and no vertices of $N_D(y_i)$ are in $A(T_n)$ or $y_i \not \in A(T_n)$. This implies that since $|N_D(y_i)| = |N_i| \ge 1$ for each $i$,
\begin{align*}
|A(T_n) \cap \{ y_1, y_2, \dots, y_k \}| + |A(T_n) \cap D| &\le |N_1| + |N_2| + \cdots + |N_k| \\
&= |D| = m - 1.
\end{align*}
By Claim 3.2, since $|V(T_n - D)| = n - m + 1 = |X_1| + |Y_1|$, the forest $T_n - D$ can be embedded to $X_1 \cup Y_1$ in $\overline{G}$ such that $y_i$ is mapped to $u_i$ for all $1 \le i \le k$. This yields a successful embedding of $T_n$ to $\overline{G}$, which is a contradiction.

\begin{case}
$|L(T_n)| \le m$.
\end{case}

By Lemma 1, there is a subset $K \subseteq A(T_n)$ such that each $w \in K$ satisfies $d_{T_n}(w) = 2$ and $w$ is not adjacent to any leaves of $T_n$, no two vertices in $K$ have a common neighbor, and
$$|K| \ge \frac{1}{4} \left( n - 8 |L(T_n)| + 12 \right).$$
Now note that
$$|K| + |L(T_n)| \ge \frac{1}{4} \left( n - 4m + 12 \right) > m + 3$$
since $n \ge m^2 - m + 1 > 8m$ for $m \ge 9$. Therefore since $|K| + |L(T_n)|$ is an integer, it holds that $|K| + |L(T_n)| \ge m + 4$. Since $|L(T_n)| \le m$, this implies that $|K| \ge 4$.

We now claim that there is a set $D \subseteq K \cup L(T_n)$ such that $|D| = m - 1$ and $d \le m - 5$ where $d$ is the maximum value $d_D(x)$ over all $x \in V(T_n) - D$. If $|K| \ge 5$, then choosing any subset of $K \cup L(T_n)$ of size $m - 1$ containing at least five vertices of $K$ yields the desired set $D$. This is because each vertex $x \in V(T_n) - D$ can be adjacent to at most one vertex in $K$ and therefore $d_D(x) \le |D| - |D \cap K| + 1 \le m - 5$. Thus it suffices to show $D$ exists when $|K| = 4$ and $|L(T_n)| = m$. Note that $|K| = 4$ implies that $V(T_n) - (K \cup L(T_n))$ contains at least four vertices since $K \cup L(T_n)$ is an independent set and no two vertices in $K$ have a common neighbor. Let $y, z \in V(T_n) - (K \cup L(T_n))$ be the two vertices adjacent to the most leaves of $T_n$. By Claim 3.1, there are at least two leaves not adjacent to $y$ and at least two not adjacent to $z$. Let $D$ consist of a leaf not adjacent to $y$, a different leaf not adjacent to $z$, the four vertices in $K$ and $m - 7$ other leaves. Assume for contradiction that $x \in V(T_n) - D$ is such that $d_D(x) \ge m - 4$. It follows that $x$ must be adjacent to one vertex in $K$ and all of the $m - 5$ leaves in $D$. Thus $x \not \in \{y, z \}$ since $y$ and $z$ are adjacent to at most $m - 6$ leaves in $D$. Now note that since each leaf of $T_n$ is adjacent to exactly one other vertex, it holds that
$$3d_{L(T_n)}(x) \le d_{L(T_n)}(x) + d_{L(T_n)}(y) + d_{L(T_n)}(z) \le |L(T_n)| \le m.$$
Therefore $d_{L(T_n)}(x) \le m/3 < m - 5$ since $m \ge 9$, which is a contradiction. This shows that the desired set $D$ exists.

As in Case 1, let $D = \{d_1, d_2, \dots, d_{m-1} \}$ and let $N_{T_n}(D) = \{ y_1, y_2, \dots, y_k \}$ where $k = |N_{T_n}(D)|$. Define $N_i$ as in Case 1 and note again that $|N_i| \le d$ for each $1 \le i \le k$. Also note that since each vertex in $D$ has degree at most two, it follows that $k \le 2m - 2$. As before, the number of vertices in $X_1$ adjacent to all vertices in $N_i$ in $\overline{G}$ is at least
$$|X_1| - |N_i| \cdot (m - 1) \ge (m - 3 - d)(m - 1) + 1 > 2m - 2 \ge k,$$
since $d \le m - 5$. Therefore, it again follows that greedily selecting $k$ vertices $u_1, u_2, \dots, u_k$ such that $u_i$ is adjacent to all vertices in $N_i$ in $\overline{G}$ and $u_i \in X_1$ for all $1 \le i \le k$ succeeds. Now note that since $K \subseteq A(T_n)$, it follows that the neighbors of vertices in $K \cap D$ are in $B(T_n)$. Thus $A(T_n) \cap \{ y_1, y_2, \dots, y_k \}$ includes only neighbors of leaves in $L(T_n) \cap D$. By the same reasoning as in the previous case, we have that
$$|A(T_n) \cap \{ y_1, y_2, \dots, y_k \}| + |A(T_n) \cap L(T_n) \cap D| \le |L(T_n) \cap D|.$$
Now it follows that
\begin{align*}
|A(T_n) &\cap \{ y_1, y_2, \dots, y_k \}| + |A(T_n) \cap D| \\
&= |A(T_n) \cap \{ y_1, y_2, \dots, y_k \}| + |A(T_n) \cap L(T_n) \cap D| + |K \cap D| \\
&\le |L(T_n) \cap D| + |K \cap D| = |D| = m - 1.
\end{align*}
By Claim 3.2, since $|V(T_n - D)| = n - m + 1 = |X_1| + |Y_1|$, the forest $T_n - D$ can be embedded to $X_1 \cup Y_1$ in $\overline{G}$ such that $y_i$ is mapped to $u_i$ for all $1 \le i \le k$. This yields a successful embedding of $T_n$ to $\overline{G}$, which is a contradiction. This completes the proof of Theorem 4.

\section{Proof of Theorem 5}

Let $UC_n$ be a particular connected graph on $n$ vertices containing a single cycle $c(UC_n)$ and let $G$ be a graph with $2n - 1$ vertices. Assume for the sake of contradiction that $G$ does not contain $F_m$ as a subgraph and $\overline{G}$ does not contain $UC_n$ as a subgraph. Our proof of Theorem 5 follows a sequence of claims similar to those used to prove Theorem 4. We begin by showing the existence of the sets $X_i, Y_i$ and $U_i$ for $i = 1, 2$ but require a more involved approach than in Theorem 4. From this point forward, we construct embeddings of the tree formed by eliminating an edge $t_1 t_2$ on $c(UC_n)$ to $\overline{G}$ while paying close attention to the images of $t_1$ and $t_2$ to ensure that these two vertices are joined by an edge in $\overline{G}$. We prove Theorem 5 in the case when $T_n$ contains a vertex adjacent to many leaves, again requiring a different method from that used to prove the corresponding claim in Theorem 4. The proof of Theorem 5 then follows a very similar structure to that of Theorem 4 with additional attention required for $t_1$ and $t_2$. Where details are similar, we refer to parts of the proof of Theorem 4. We now begin our proof of Theorem 5.

As proven by Shi in \cite{shi2010ramsey}, $R(C_n, F_m) = 2n - 1$ for all $n > 3m$. Therefore we may assume that $UC_n$ is not the cycle $C_n$. If $UC_n$ is not a cycle, there must be an edge $t_1 t_2$ of $c(UC_n)$ such that removing the edge $t_1 t_2$ yields a tree in which $t_1$ is not a leaf. Let $T_n$ be the tree formed on removing the edge $t_1 t_2$ from $UC_n$. As in the proof of Theorem 4, define the sets $U_v, X_v$ and $Y_v$ for any $v \in V(G)$. Since $G$ does not contain $F_m$ as a subgraph, the same inequalities (3.1)--(3.4) on the sizes $|U_v|, |X_v|$ and $|Y_v|$ as in the proof of Theorem 4 hold here. In the proof of Theorem 5, we often explicitly exclude $t_2$ from sets of leaves even though this is only necessary in the case that $t_2$ is a leaf of $T_n$.

Our first two claims show the existence of analogues of the sets $X_1, Y_1$ and $U_1$ from our proof of Theorem 4. The first claim we prove implicitly involves the fact that any graph $H$ satisfying $\delta(H) \ge |V(H)|/2$ is pancyclic -- contains a cycle of every length between $3$ and $|V(H)|$, inclusive. We prove this fact using an argument of Droogendijk \cite{1158044}.

\begin{claim}
We have that $\Delta(G) \ge n$.
\end{claim}

\begin{proof}
Assume for contradiction that $\Delta(G) \le n - 1$. We will consider the two cases in which $\Delta(G) \le n - 2$ and $\Delta(G) = n - 1$, separately. We first consider the case wherein $\Delta(G) \le n - 2$, which implies that $\delta(\overline{G}) \ge |V(G)| - 1 - (n - 2) = n$. Let $v \in V(G)$ be an arbitrary vertex and consider the graph $\overline{G} - v$. Now note that $\delta(\overline{G} - v) \ge n - 1 = |V(\overline{G} - v)|/2$, which implies by Dirac's Theorem that $\overline{G} - v$ has a Hamiltonian cycle. Let $u_1, u_2, \dots, u_{2n - 2}$ be the vertices of $\overline{G} - v$ ordered along this Hamiltonian cycle and let $k = |c(UC_n)|$. Since $v$ is adjacent in $\overline{G}$ to at least $n$ vertices in $\overline{G} - v$, it must be adjacent to some two of the form $u_i$ and $u_{i + k - 2}$ where indices are taken modulo $2n - 2$. The vertices $v, u_i, u_{i+1}, \dots, u_{i + k - 2}$ now form a cycle $C$ of length $k = |c(UC_n)|$.

Consider the following procedure to embed $UC_n$ to $\overline{G}$. Begin by mapping the cycle $c(UC_n)$ to $C$ arbitrarily. At any point in this procedure, let $H$ denote the subgraph of $\overline{G}$ that has so far been embedded to. Note that $H = C$ initially and $C \subseteq H$ always in this procedure. Greedily extend the embedding as follows: if $u_1 \in V(UC_n)$ has been mapped to $u_2 \in V(H) \subseteq V(G)$, $w \in V(UC_n)$ has not been embedded to $\overline{G}$ and $w$ is adjacent to $u_1$ in $UC_n$, map $w$ to some element of $\overline{N}_{V(G) - V(H)}(u_2)$ if it is non-empty. Note that each possible $w$ has a unique neighbor among the vertices of $V(UC_n)$ that have been embedded to $H$ because the unique cycle of $UC_n$ is $c(UC_n)$, which at the start of the procedure has already been embedded to $H$. Furthermore, if not all of $UC_n$ has been embedded to $\overline{G}$ then some such $w \in V(UC_n)$ must exist since each vertex of $UC_n$ is connected in $UC_n - E(c(UC_n))$ to a vertex in $c(UC_n)$, which begins embedded to $H$. Since $\overline{G}$ does not contain $UC_n$ as a subgraph, the embedding must fail with $\overline{N_{V(G) - V(H)}}(u_2) = \emptyset$ for some $u_2 \in V(H)$ where $|V(H)| \le n - 1$. This implies that $d(u_2) \ge |V(G) - V(H)| \ge n$, which contradicts $\Delta(G) \le n - 2$.

Now consider the case in which $\Delta(G) = n - 1$. Let $u \in V(G)$ be such that $d(u) = n - 1$ and note that $|U_u| \ge d(u) - 2m + 2 = n - 2m + 1$. We now show that $\overline{G}$ contains a cycle $C_k$ of length $k = |c(UC_n)|$. Consider the case $k > n - 2m + 1$. As proven by Shi in \cite{shi2010ramsey}, $R(C_\ell, F_m) = 2\ell - 1$ for $\ell > 3m$. Since $2k - 1 \le 2n - 1 = |V(G)|$, $k > n - 2m + 1 \ge 3m$ and $G$ does not contain $F_m$, it follows that $\overline{G}$ contains $C_k$ as a subgraph. If $k \le n - 2m + 1$, then $U_u$ is a complete graph in $\overline{G}$ of size at least $k$ and thus contains $C_k$.

Now apply the same embedding procedure as used in the previous case to embed $UC_n$ to $\overline{G}$, beginning by mapping $c(UC_n)$ to the $C_k$ subgraph of $\overline{G}$. Since $\overline{G}$ does not contain $UC_n$ as a subgraph, this yields that $d(u_2) \ge n$ for some $u_2 \in V(G)$ as in the previous case, contradicting $\Delta(G) = n - 1$ and proving the claim.
\end{proof}

Let $v \in V(G)$ be a vertex satisfying $d(v) \ge n$, which exists by Claim 4.1. Now define $X_1, Y_1$ and $U_1$ as in the proof of Theorem 4. We next prove two claims that show the existence of $X_2, Y_2$ and $U_2$, the second of which resembles Claim 3.3 in the proof of Theorem 4. However, a more involved proof accounting for $UC_n$ will be needed. We remark that this proof also implies Claim 3.3 in the proof of Theorem 4. In order to provide a simpler proof of Theorem 4, we did not use this argument for Claim 3.3. Let $O_1 = V(G) - (X_1 \cup Y_1)$.

\begin{claim}
Suppose it holds that $\overline{d_{O_1}}(u) \ge m - 1$ for all $u \in X_1 \cup Y_1$. Let $K \subseteq X_1 \cup Y_1$ and $H \subseteq O_1$ be such that $|H| \ge n + 1$ and $2|K| \ge |H| + 1$. Then there exists three distinct vertices $x, y, z$ with $x, z \in K$ and $y \in H$ such that $y$ is adjacent to $x$ and $z$ in $\overline{G}$.
\end{claim}

\begin{proof}
Assume for contradiction that no such $x, y$ and $z$ exist. It follows that each $x \in H$ is adjacent to at most one vertex of $K$ in $\overline{G}$ and therefore that the number of edges between $K$ and $H$ in $\overline{G}$ is at most $|H|$. Note that for each $y \in K \subseteq X_1 \cup Y_1$,
$$\overline{d_{H}}(y) \ge \overline{d_{O_1}}(y) - |O_1 - H| \ge m - 1 - (n + m - 2 - (n + 1)) = 2.$$
The number of edges between $K$ and $H$ in $\overline{G}$ is therefore at least $2|K| \ge |H| + 1$, which is a contradiction.
\end{proof}

\begin{claim}
There is a vertex $u \in X_1 \cup Y_1$ such that $d_{O_1}(u) \ge n$.
\end{claim}

\begin{proof}
Assume for contradiction that $d_{O_1}(u) \le n - 1$ for all $u \in X_1 \cup Y_1$. Note that $|O_1| = 2n - 1 - |X_1 \cup Y_1| = n + m - 2$. Thus for all $u \in X_1 \cup Y_1$,
$$\overline{d_{O_1}}(u) = |O_1| - d_{O_1}(u) \ge n + m - 2 - (n - 1) = m - 1.$$

If $T_n$ has at least $m$ leaves then let $L \subseteq L(T_n)$ be a set of $|L| = m - 1$ of these leaves such that $t_2 \not \in L$, and let $x_1, x_2, \dots, x_k$ be the vertices adjacent to $L$ in $T_n$. Let $L = L_1 \cup L_2 \cup \cdots \cup L_k$ where $L_i$ consists of the leaves in $L$ adjacent to $x_i$ for each $1 \le i \le k$. Consider the following embedding of $T_n$ into $\overline{G}$. By Claim 3.2, the tree $T_n - L$ on $n - m + 1 = |X_1| + |Y_1|$ vertices can be embedded into $\overline{G}$ such that $t_1$ and $t_2$ are mapped to vertices in $X_1$. Now let $v_1, v_2, \dots, v_k \in X_1 \cup Y_1$ denote the vertices that $x_1, x_2, \dots, x_k \in V(T_n - L)$ were embedded to. Now greedily select $|L_i|$ distinct neighbors of $v_i$ in $O_1$ in $\overline{G}$ that have not been selected by any $v_j$ where $j < i$ to embed the vertices of $L_i$ to for $1 \le i \le k$. This is possible since $|L_1| + |L_2| + \cdots + |L_k| = m - 1$ and $\overline{d_{O_1}}(v_i) \ge m - 1$ for $1 \le i \le k$, yielding a successful embedding of $T_n$ into $\overline{G}$. Now note that since $t_1$ and $t_2$ were embedded to $X_1$, which is a complete graph in $\overline{G}$, it follows that $\overline{G}$ contains $UC_n$ as a subgraph, which is a contradiction. Therefore $T_n$ must have at most $m - 1$ leaves.

By Lemma 1, if $F = \{t_1, t_2 \}$ then there is a subset $D \subseteq A(T_n)$ such that $F \cap D = \emptyset$, each $v \in D$ satisfies that $d_{T_n}(v) = 2$ and $v$ is not adjacent to a leaf of $T_n$, no two vertices in $D$ have a common neighbor in $T_n$, and $|D| \ge ( n - 8 |L(T_n)| + 8 )/4$ where $|L(T_n)| \le m - 1$ is the number of leaves of $T_n$. Let $L = L(T_n) - F \cap L(T_n)$ where either $F \cap L(T_n) = \{ t_2 \}$ or $F \cap L(T_n) = \emptyset$. Since $T_n$ has at least two leaves when $n \ge 2$, it follows that $|L(T_n)| \ge 2$ and $|L| \ge 1$. Because $n \ge m^2 -m + 1 \ge 8m - 12$ for $m \ge 8$, we have
\begin{align*}
|L| + |D| &\ge |L(T_n)| - 1 + (n - 8|L(T_n)| + 8)/4 \\
&\ge (n - 4m + 8)/4 > m - 1.
\end{align*}
Let $d_1, d_2, \dots, d_k$ be $k$ elements of $D$ where $k = m - 1 - |L| \le m - 2$ and where the two neighbors of $d_i$ in $T_n$ are $x_i$ and $y_i$ for each $1 \le i \le k$. Note that the vertices $x_i$ and $y_i$ for $1 \le i \le k$ are all pairwise distinct.

Now consider the following embedding of $T_n$ into $\overline{G}$. We begin by greedily embedding the vertices $d_i$ to $k$ distinct vertices in $O_1$ and the vertices $x_i$ and $y_i$ to $2k$ distinct vertices in $X_1$. On embedding $d_i$, exactly $i - 1$ vertices have been embedded to $O_1$. Therefore $N_H = |O_1| - (i - 1) = n + m - 2 - (i - 1) \ge n + 1$ vertices of $O_1$ have not been embedded to since $i \le k \le m - 2$. Furthermore, exactly $2(i-1)$ vertices have been embedded to $X_1$ and $N_K = |X_1| - 2(i-1) \ge n - 3m + 5 - 2i$ vertices of $X_1$ have not been embedded to. Now note that since $n \ge m^2 -m + 1 > 10m - 16 \ge 7m - 10 + 3i$ for $m \ge 10$ because $i \le m - 2$,
$$2N_K = 2n - 6m + 10 - 4i \ge n + m - i = N_H + 1.$$
Thus by Claim 4.2 there are vertices $a_i, b_i, c_i$ that have not been embedded to, with $a_i \in O_1$ and $b_i, c_i \in X_1$, to embed $x_i, d_i, y_i$ to, respectively. Note that since $x_i$ and $y_i$ are adjacent to $d_i$ in $T_n$ and $d_i \in A(T_n)$, it follows that $x_i, y_i \in B(T_n)$ for all $1 \le i \le k$. Therefore the $k$ vertices $d_1, d_2, \dots, d_k$ are the vertices of $A(T_n)$ that have so far been embedded to $\overline{G}$. Since $|A(T_n)| \ge n/2$, $n \ge m^2 - m + 1 > 8m$ for $m \ge 9$ and $k + |L| = m - 1$, at least
$$|A(T_n)| - k - |L| \ge n/2 - m + 1 > 3m + 1 \ge |Y_1| + 2$$
vertices of $A(T_n)$ that are not in $L$ have not been embedded to $\overline{G}$. Embed $|Y_1|$ of these vertices arbitrarily to $Y_1$, excluding $t_1$ and $t_2$ if either is in $A(T_n)$ and not in $D$, $\{ x_1, x_2, \dots, x_k\}$ or $\{ y_1, y_2, \dots, y_k\}$. Embed the remaining
$$n - |L| - 3k - |Y_1| = n - m + 1 - |Y_1| - 2k = |X_1| - 2k$$
vertices in $T_n - L$ arbitrarily to $X_1 - \{ b_1, \dots, b_k, c_1, \dots, c_k\}$. Since the vertices $\{d_1, d_2, \dots, d_k \}$ form an independent set in $T_n$, the vertices embedded to $Y_1$ are a subset of $A(T_n)$ and therefore also independent in $T_n$, and each vertex in $X_1$ is adjacent to every vertex of $X_1 \cup Y_1$ other than itself, this is a successful embedding of $T_n - L$ into $\overline{G}$. 

Now note that since no $v \in D$ is adjacent to a leaf in $T_n$, all neighbors of $L$ have been embedded to vertices in $X_1 \cup Y_1$. Apply the same greedy procedure as in the previous embedding to embed the leaves in $L$ to distinct vertices in $O_1 - \{a_1, a_2, \dots, a_k \}$ such that the embedded leaves are adjacent to their embedded neighbors in $X_1 \cup Y_1$. This is possible since $|L| + k = m - 1$ and $\overline{d_{O_1}}(u) \ge m - 1$ for all $u \in X_1 \cup Y_1$. This yields a successful embedding of $T_n$ into $\overline{G}$. Now note that $t_1$ and $t_2$ are necessarily mapped under this embedding to vertices in $X_1$ and therefore are adjacent in $\overline{G}$. Thus $\overline{G}$ contains $UC_n$ as a subgraph, which is a contradiction.
\end{proof}

Now define $X_2, Y_2$ and $U_2$ as in the proof of Theorem 4 with the vertex $u$ guaranteed by Claim 4.3. Also define the set $W = V(G) - (U_1 \cup U_2)$. We next prove a claim that is a weaker analogue of Claim 3.1 for unicyclic graphs. The beginning of the proof uses ideas similar to those used in Claim 3.1.

\begin{claim}
Each vertex $x \in V(T_n)$ adjacent to at most $2m - 2$ leaves of $T_n$ and $\Delta(T_n) < 5n/9$.
\end{claim}

\begin{proof}
First we show each vertex of $T_n$ is adjacent to at most $2m - 2$ leaves. Assume for the sake of contradiction that some $x \in V(T_n)$ is adjacent to at least $2m - 1$ leaves. By Theorem 3, $\overline{G}$ contains $S_n$ as a subgraph since $G$ does not contain $F_m$. Let $w$ be the center of this $S_n$ subgraph. If $w \in U_1$, consider the following embedding of $UC_n$ to $\overline{G}$. Note that $\overline{d_{V(G) - U_1}}(w) \ge \overline{d}(w) - (|U_1|-1) \ge n - |U_1|$. Now map $x$ to $w$, map $n - |U_1|$ leaves adjacent to $x$ excluding $t_2$ to distinct vertices in $\overline{N_{V(G) - U_1}}(w)$ and map the remaining $|U_1| - 1$ vertices of $T_n$ arbitrarily to $U_1 - \{ w \}$. This is possible because $2m - 2 \ge n - |U_1|$ for $m \ge 2$. Since $U_1$ is complete in $\overline{G}$ and both $t_1$ and $t_2$ are mapped to vertices in $U_1$, this embedding succeeds and yields a contradiction. Therefore we may assume that $w \not \in U_1$ and, by a symmetric argument, that $w \not \in U_2$. Thus we have that $w \in W$.

Now note that if $w \in W$, then since $\overline{d}(w) \ge n - 1$ and $|W| \le 4m - 3$,
$$\overline{d_{U_1}}(w) + \overline{d_{U_2}}(w) \ge \overline{d}(w) - (|W|-1) \ge n - 4m + 3.$$
Since $n - 4m + 3 \ge 7$, $w$ is adjacent to at least four vertices of one of $U_1$ and $U_2$ in $\overline{G}$. Without loss of generality, assume $w$ is adjacent to at least four vertices of $U_1$ in $\overline{G}$ with $\overline{d_{U_1}}(w) \ge 4$.

Let $C_1, C_2, \dots, C_d$ be the connected components of the forest $T_n - x$ satisfying that either $|C_i| \ge 2$ or $t_2 \in C_i$ for each $1 \le i \le d$. Furthermore, let $y_1, y_2, \dots, y_d$ denote the unique neighbors of $x$ in $C_1, C_2, \dots, C_d$, respectively, in $T_n$. Since $x$ is adjacent to at least $2m - 2$ leaves of $T_n$ excluding $t_2$, it follows that
$$|C_1 \cup C_2 \cup \cdots \cup C_d| \le (|V(T_n)| - 1) - (2m - 2) = n - 2m + 1.$$
Furthermore, because there is at most one index $i$ satisfying $t_2 \in C_i$ and $|C_j| \ge 2$ for all $j \neq i$, we have that
$$d \le \frac{1}{2} (1 + |C_1 \cup C_2 \cup \cdots \cup C_d|) \le (n - 2m + 2)/2.$$
Note that since $t_1$ is not a leaf, $t_1 \in C_j$ for some $1 \le j \le d$. Without loss of generality, we may assume that $t_1, t_2 \in C_1 \cup C_2 \cup \{ x \}$. Now consider the following embedding of $UC_n$ to $\overline{G}$. Map $x$ to $w$ and map $y_1, y_2, \dots, y_t$ to distinct vertices in $\overline{N_{U_1}}(w)$ where $t = \overline{d_{U_1}}(w) - 2$. Map $t_1$ and $t_2$ to two elements of $\overline{N_{U_1}}(w)$ that have not yet been embedded to, if either $t_1$ or $t_2$ has not yet been embedded. Furthermore, map $C_1 \cup C_2 \cup \cdots \cup C_t - \{ y_1, y_2, \dots, y_t \} - \{t_1, t_2 \}$ to distinct vertices in $U_1$ that have not yet been embedded to. Note that since $w$ is adjacent to at least four vertices of $U_1$ in $\overline{G}$, it follows that $t \ge 2$ and that $C_1 \cup C_2$ has been embedded to $U_1$. Similarly, map $y_{t+1}, y_{t+2}, \dots, y_d$ to distinct vertices in $\overline{N_{U_2}}(w)$ and $C_{t+1} \cup C_{t+2} \cup \cdots \cup C_d - \{y_{t+1}, y_{t+2}, \dots, y_d\}$ to distinct vertices in $U_2$ that have not yet been embedded to. The mapping described above is possible because
\begin{align*}
d \le (n - 2m + 2)/2 &\le n - 4m + 1 \le \overline{d_{U_1}}(w) + \overline{d_{U_2}}(w) - 2, \\
|C_1 \cup C_2 \cup \cdots \cup C_t| &\le |C_1 \cup C_2 \cup \cdots \cup C_d| \le n - 2m + 1 < |U_1|,\\
|C_{t+1} \cup C_{t+2} \cup \cdots \cup C_d| &\le |C_1 \cup C_2 \cup \cdots \cup C_d| \le n - 2m + 1 < |U_2|,
\end{align*}
since $n \ge m^2 - m + 1 \ge 6m$ for $m \ge 7$. Now embed the remaining vertices of $T_n$, all of which are leaves adjacent to $x$ to distinct neighbors of $w$ in $\overline{G}$ that have not already been embedded to. This is possible because $\overline{d}(w) \ge n - 1$. This yields a successful embedding of $T_n$ to $\overline{G}$ since $U_1$ and $U_2$ induce complete subgraphs of $\overline{G}$. Now note that if $x \in \{ t_1, t_2 \}$, the other vertex of $\{ t_1, t_2 \}$ has been embedded to a neighbor of $w$ in $\overline{G}$. If not, then $t_1$ and $t_2$ were both embedded to vertices of $U_1$. In either case, the vertices $t_1$ and $t_2$ were mapped to are adjacent in $\overline{G}$, which implies that $UC_n$ is a subgraph of $\overline{G}$. This is a contradiction and therefore each vertex of $T_n$ is adjacent to at most $2m - 2$ leaves.

Now we show that $\Delta(T_n) < 5n/9$. Assume for contradiction that $\Delta(T_n) \ge 5n/9$ and let $x$ be a vertex of largest degree in $T_n$ where $d_{T_n}(x) \ge 5n/9$. Let $K \subseteq L(T_n)$ be the set of leaves adjacent to $x$ and let $H \subseteq V(T_n)$ be the set of vertices of degree at least two adjacent to $x$. It follows that $d_{T_n}(x) = |K| + |H|$ and $n = |V(T_n)| \ge 1 + |K| + 2|H|$ since $T_n$ is a tree. Therefore we have that
$$|K| + n \ge 1 + 2|K| + 2|H| = 1 + 2d_{T_n}(u) \ge 1 + 10n/9,$$
which implies that $|K| \ge 1 + n/9 \ge 2m - 1$ since $n \ge m^2 -m + 1 \ge 18m - 18$ for $m \ge 18$. Applying the previous argument now yields a contradiction.
\end{proof}

We now prove a claim specializing the result in Lemma 2 as needed to account for $t_1$ and $t_2$. Claim 4.5 has no analogue in the proof of Theorem 4 and will be crucial in proving Claims 4.7 and 4.8.

\begin{claim}
There is a vertex $x \in V(T_n)$ such that the vertices of the forest $T_n - x$ can be partitioned into two sets $H$ and $J$ such that there are no edges between $H$ and $J$ in $T_n$, we have the inequalities
$$2m - 2 \le |H|, \, |J| \le n - 2m + 1 \quad \text{and} \quad d_{H}(x) \le 2m - 2$$
and one of the following holds:
\begin{enumerate}
\item $x \in \{ t_1, t_2 \}$;
\item $x$ is adjacent to neither $t_1$ nor $t_2$;
\item $\{ t_1, t_2 \} \subseteq H$; or
\item $\{ t_1, t_2 \} \subseteq J$.
\end{enumerate}
\end{claim}

\begin{proof}
By Lemma 2, there is a vertex $x$ of $T_n$ such that $T_n - x$ can be partitioned into two sets $P$ and $Q$ such that there are no edges between $P$ and $Q$ and $(n-1)/3 \le |P|, \, |Q| \le 2(n-1)/3$. Let $C_1, C_2, \dots, C_d$ be the connected components of the forest $T_n - x$. Since each of $P$ and $Q$ is a union of connected components $C_i$, it follows that $|C_i| \le 2(n-1)/3 \le n - 2m + 1$ for all $1 \le i \le d$, since $n \ge m^2 -m + 1 \ge 6m - 5$ for $m \ge 6$. Furthermore, we may assume without loss of generality that $P = C_1 \cup C_2 \cup \dots \cup C_t$ and $Q = C_{t+1} \cup C_{t+2} \cup \dots \cup C_d$.

Let $H = C_1 \cup C_2 \cup \cdots \cup C_p$ where $p$ is the minimum positive integer such that $|C_1| + |C_2| + \cdots + |C_p| \ge 2m - 2$. Note that since $|P| \ge (n-1)/3 \ge 2m - 2$, it follows that $p \le t$. Furthermore, because $|C_i| \ge 1$ for all $1 \le i \le d$ it follows that $d_H(x) = p \le 2m - 2$. Set $J = C_{p+1} \cup C_{p+2} \cup \cdots \cup C_d$. For these sets $H$ and $J$ to not have the desired properties (1)--(4), it must hold that $x \not \in \{t_1, t_2 \}$, $x$ is adjacent to at least one of $t_1$ and $t_2$ and exactly one of $t_1$ and $t_2$ is in each of $H$ and $J$.

Assume that $t_1 \in C_i$ and $t_2 \in C_j$ where $1 \le i \le p$ and $p + 1 \le j \le d$. The case in which $p + 1 \le i \le d$ and $1 \le j \le p$ can be handled by a symmetric argument. If $2m - 2 \le |C_i \cup C_j| \le n - 2m + 1$, then setting $H = C_i \cup C_j$ and $J = V(T_n - x) - H$ yields sets satisfying the desired properties since $d_H(x) = 2$ and $m \ge 2$. Now consider the case in which $|C_i \cup C_j| < 2m - 2$. This yields the sequence of inequalities,
\begin{align*}
2m - 2 &\le |C_1| + |C_2| + \cdots + |C_p| \\
&\le |C_1| + |C_2| + \cdots + |C_p| + |C_i| + |C_j| \\
&\le |P| + |C_i \cup C_j| < 2(n-1)/3 + 2m - 2 \le n - 2m + 1
\end{align*}
since $n \ge m^2 - m + 1 \ge 12m - 11$ for $m \ge 12$. Note that $C_1 \cup C_2 \cup \cdots \cup C_p \subseteq P$ since $p \le t$. If $|C_k| = |C_l| = 1$ for some $1 \le k < l \le p$ then set $H$ to be the union of all $C_q$ where $1 \le q \le p$ and $q \not \in \{ k, l \}$ with $C_i$ and $C_j$. Also set $J = V(T_n - x) - H$. In this case, we have that since $|C_i|, |C_j| \ge |C_k| = |C_l| = 1$,
\begin{align}
2m -2 &\le |C_1| + |C_2| + \cdots + |C_p| \le |H| \\
&\le |C_1| + |C_2| + \cdots + |C_p| + |C_i| + |C_j| \le n - 2m + 1.
\end{align}
It also follows that $d_H(x) \le p \le 2m - 2$. Note that $d_H(x) < p$ if either $i$ or $j$ is at most $p$ and not in $\{k, l\}$. This yields valid sets $H$ and $J$. If no such $k$ and $l$ exist, set $H = C_1 \cup C_2 \cup \cdots \cup C_p \cup C_i \cup C_j$ and again set $J = V(T_n - x) - H$. Since no such $k$ and $l$ exist, $|C_k| = 1$ for at most one $1 \le k \le p$ and $|C_l| \ge 2$ for all $1 \le l \le p$ with $l \neq k$. Since $p$ is the minimal $p$ for which $|C_1| + |C_2| + \cdots + |C_p| \ge 2m - 2$, it necessarily follows that $p \le m - 1$. Therefore we have that $d_H(x) \le p + 2 \le 2m - 2$. In this case, we also have exactly the sequence of inequalities (4.1) and (4.2). This again yields valid sets $H$ and $J$, showing that the inequality $|C_i \cup C_j| < 2m - 2$ cannot hold.

Therefore we may assume that $|C_i \cup C_j| \ge n - 2m + 2$. We consider the case in which $x$ is adjacent to $t_1$. The case in which $x$ is adjacent to $t_2$ follows by a symmetric argument, because we do not use the fact that $t_1$ is not a leaf. Note that either $|C_i| \ge n/2 - m + 1$ or $|C_j| \ge n/2 - m + 1$. If $|C_i| \ge n/2 - m + 1$, consider the connected components of the forest $T_n - t_1$, some subset of which has as its union $C_i - \{ t_1 \}$. Note that $C_i \subseteq P$ since $i \le p \le t$. Setting $H$ to be the smallest union of connected components of $T_n - t_1$ that are subsets of $C_i$ satisfying that $|H| \ge 2m - 2$ and setting $J = V(T_n - t_1) - H$ yields valid sets by the same argument as in the first construction described above because
$$(n-1)/3 \le n/2 - m \le |C_i - \{ t_1 \}| \le |P| \le 2(n-1)/3.$$

If $|C_j| \ge n/2 - m + 1$, let $y$ be the unique neighbor of $x$ in $C_j$. Note that we may assume that $y \neq t_2$ since otherwise the construction above replacing $t_1$ with $t_2$ yields valid sets $H$ and $J$. Consider the connected components of the forest $T_n - y$ and let $K$ be the connected component of $t_2$. Let $M$ be the union of $\{ x \}$ and all $C_q$ such that $q \neq j$. Note that $M$ is a connected component of $T_n - y$ and that $t_1 \in M$. If $|K| \le |C_j| - 2m + 2$, then set $H = M \cup K$ and $J = C_j - (\{ y \} \cup K)$. Note that this implies $d_H(y) = 2$. It follows that $|H| \le n - 2m + 1$ and that $|J| < |C_j| \le \max\{|P|, |Q| \} \le 2(n-1)/3 \le n - 2m + 1$ since $n \ge m^2 -m +1 \ge 6m - 5$ for $m \ge 6$. Note that $t_1 \in M$ and $t_2 \in K$, implying that these are valid sets $H$ and $J$. Thus we may assume that $|K| \ge |C_j| - 2m + 3 \ge n/2 - 3m + 4$. 

Note that $y$ is not adjacent to $t_1$ since otherwise $y, x$ and $t_1$ would form a triangle in $T_n$. Now if $y$ is not adjacent to $t_2$ then $y$ satisfies property (2). In this case, define $H$ to be the smallest union of connected components of $T_n - y$ that are subsets of $C_j$ satisfying that $|H| \ge 2m - 2$ and set $J = V(T_n - y) - H$, which yields valid sets using the argument above and the fact that $C_j \subseteq Q$ implies the inequalities
$$(n-1)/3 \le n/2 - m \le |C_j - \{ y \}| \le |Q| \le 2(n-1)/3.$$
If $y$ is adjacent to $t_2$, then consider the connected components of $T_n - t_2$. Note that there is one connected component that contains $y$ and the others are all subsets of $K - \{ t_2 \}$. Let $H$ be the smallest union of connected components of $T_n - t_2$ that are subsets of $K - \{t_2 \}$ satisfying that $|H| \ge 2m - 2$ and set $J = V(T_n - t_2)$. This yields valid sets by combining the previous argument above with the inequalities
$$|K| - 1 \le |C_j| - 1 \le 2(n-1)/3 - 1 \le n - 2m +1 \quad \text{and}$$
$$|K| - 1 \ge n/2 -3m + 3 \ge 2m - 2,$$
which hold since $n \ge m^2 - m + 1 \ge 10m - 10$ for $m \ge 10$. This completes the proof of the claim.
\end{proof}

We now consider the case in which there is a matching on two edges in $\overline{G}$ between $U_1$ and $U_2$. We consider the case in which this matching does not exist in Claim 4.8. The next three claims are similar to Claims 3.4, 3.5 and 3.6, they show that $U_1$ and $U_2$ induce a nearly complete bipartite subgraph of $G$ and that each vertex in $W$ has few neighbors in a large subset of one of $U_1$ or $U_2$.

\begin{claim}
If there are two disjoint edges in $\overline{G}$ between $U_1$ and $U_2$, then for each $w \in U_1$ we have $\overline{d_{U_2}}(w) \le 2m - 2$, and for each $w \in U_2$ we have $\overline{d_{U_1}}(w) \le 2m - 2$.
\end{claim}

\begin{proof}
Assume for contradiction that for some $w \in U_1$, it holds that $\overline{d_{U_2}}(w) \ge 2m - 1$. Note that there is a vertex $x \in V(T_n)$ such that the vertices of the forest $T_n - x$ can be partitioned into two sets $H$ and $J$ satisfying the conditions described in Claim 4.5. Since there are two disjoint edges in $\overline{G}$ between $U_1$ and $U_2$, there must be some edge $yz \in E(\overline{G})$ such that $y \in U_1$, $z \in U_2$ and $y \neq w$.

Consider the case in which $x \not \in \{t_1, t_2 \}$ and exactly one of $t_1$ and $t_2$ is in each of the sets $H$ and $J$. For now assume that $t_1 \in H$ and $t_2 \in J$. The case in which $t_1 \in J$ and $t_2 \in H$ can be handled with a symmetric argument. Claim 4.5 implies that $x$ is adjacent to neither $t_1$ nor $t_2$ in $T_n$. Now consider the following embedding of $UC_n$ into $\overline{G}$. Map $x$ to $w$, $t_1$ to $z$ and $t_2$ to $y$. Map the $d_H(x) \le 2m - 2$ neighbors of $x$ in $H$ to distinct vertices in $\overline{N_{U_2}}(w)$ excluding $y$ if $y \in \overline{N_{U_2}}(w)$. Note that these neighbors do not include $t_2$. Map the remaining $|H| - d_H(x) - 1$ vertices of $H$ to distinct vertices in $U_2$ that have not already been embedded to. Since $|H| \le n - 2m + 1 < |U_2|$, this is possible. Map $|J| - 1$ vertices of $J - \{ t_2 \}$ to distinct vertices in $U_1 - \{w, y \}$. This is possible because $|J| -1 \le n - 2m \le |U_1| - 2$. This yields a successful embedding of $T_n$ to $\overline{G}$ and, since $yz \in E(\overline{G})$, this also yields a successful embedding of $UC_n$ to $\overline{G}$.

Now consider the case in which $x \in \{ t_1, t_2 \}$. If $\{ t_1, t_2 \} - \{ x \}$ is in $H$, then map $x$ to $w$, $J$ to a subset of $U_1 - \{ w \}$, the $d_H(x) \le 2m - 2$ neighbors of $x$ in $H$ and $\{t_1, t_2\} - \{ x \}$ to distinct vertices in $\overline{N_{U_2}}(w)$ and the remainder of $H$ to vertices in $U_2$ that have not been embedded to. This embedding is valid by the same inequalities as above. Furthermore, this ensures that $t_1$ and $t_2$ are adjacent in $\overline{G}$. Similarly if $x \in \{ t_1, t_2 \}$ and $\{ t_1, t_2 \} - \{ x \}$ is in $J$ or if $\{t_1, t_2 \}$ is a subset of either $H$ or $J$, then this same embedding without mapping $\{t_1, t_2\} - \{ x \}$ to a vertex in $\overline{N_{U_2}}(w)$ succeeds. This again yields a successful embedding of $UC_n$ to $\overline{G}$, which is a contradiction.

A symmetric argument shows that for each $w \in U_2$ we have $\overline{d_{U_1}}(w) \le 2m - 2$.
\end{proof}

\begin{claim}
If there is an edge in $\overline{G}$ between $U_1$ and $U_2$, then for each $w \in W$ either $\overline{d_{U_1}} (w) < 1 + 5n/18$ or $\overline{d_{U_2}}(w) < 1 + 5n/18$.
\end{claim}

\begin{proof}
Let $yz \in E(\overline{G})$ where $y \in U_1$ and $z \in U_2$. Assume for contradiction that for some $w \in W$ it follows that $\overline{d_{U_1}} (w) \ge 1 + 5n/18$ and $\overline{d_{U_2}}(w) \ge 1 + 5n/18$. There is a vertex $x \in V(T_n)$ such that the vertices of the forest $T_n - x$ can be partitioned into two sets $H$ and $J$ satisfying the conditions described in Claim 4.5.

First we consider the case in which $d_{J}(x) \le \overline{d_{U_2}}(w) - 1$. Note that $d_{H}(x) \le 2m - 2 \le 5n/18 \le \overline{d_{U_1}}(w) - 1$ since $n \ge m^2 - m + 1 \ge 36(m - 1)/5$ for $m \ge 8$. Now consider the following embedding of $UC_n$ to $\overline{G}$. Map $x$ to $w$ and do one of the following depending which of (1)--(4) is true from Claim 4.5:
\begin{itemize}
\item if $x \in \{ t_1, t_2 \}$ and $\{ t_1, t_2 \} - \{ x \}$ is in $H$, then map $\{ t_1, t_2 \} - \{ x \}$ to a vertex in $\overline{N_{U_1}}(w)$;
\item if $x \in \{ t_1, t_2 \}$ and $\{ t_1, t_2 \} - \{ x \}$ is in $J$, then map $\{ t_1, t_2 \} - \{ x \}$ to a vertex in $\overline{N_{U_2}}(w)$; or
\item if $x \not \in \{ t_1, t_2 \}$ and exactly one of $t_1$ and $t_2$ is in each of $H$ and $J$, then map $\{ t_1, t_2 \} \cap H$ to $y$ and $\{ t_1, t_2 \} \cap J$ to $z$.
\end{itemize}
Note that in the last case, Claim 4.5 guarantees that $x$ is adjacent to neither $t_1$ nor $t_2$. After this step, at most one vertex has been mapped to each of $U_1$ and $U_2$. Thus $x$ has either $d_H(x)$ or $d_H(x) - 1$ neighbors in $H$ and either $d_J(x)$ or $d_J(x) - 1$ neighbors in $J$ that have not already been embedded to $\overline{G}$. Now map the remaining neighbors of $x$ in $H$ to vertices in $\overline{N_{U_1}}(w)$ that have not already been embedded to. Similarly map the neighbors of $x$ in $J$ to vertices in $\overline{N_{U_2}}(w)$ that have not already been embedded to. This is possible because $d_H(x) \le \overline{d_{U_1}}(w) - 1$ and $d_{J}(x) \le \overline{d_{U_2}}(w) - 1$. Map the remaining vertices in $H$ to distinct vertices of $U_1$ that have not yet been embedded to and map the remaining vertices in $J$ to distinct vertices of $U_2$ that have not yet been embedded to. This is possible because $|H| \le n - 2m + 1 < |U_1|$ and $|J| \le n - 2m + 1 < |U_2|$. Since $t_1$ and $t_2$ are either mapped to $y$ and $z$, two vertices in $U_1$, two vertices in $U_2$, $x$ and a vertex in $\overline{N_{U_1}}(w)$ or $x$ and a vertex in $\overline{N_{U_2}}(w)$, it follows that $t_1$ and $t_2$ are adjacent in $\overline{G}$. Therefore $UC_n$ is a subgraph of $\overline{G}$, which is a contradiction.

Now suppose that $d_J(x) \ge \overline{d_{U_2}}(w) \ge 1 + 5n/18$. Note that $J$ is the union of $d_J(x)$ connected components of $T_n - x$, at least $d_J(x) - 2$ of which contain neither $t_1$ nor $t_2$. Let $C$ be the union of $k = d_J(x) - \overline{d_{U_2}}(w) + 1$ of these $d_J(x) - 2$ connected components. Let $J' = J - C$ and let $H' = H \cup C$. Note that $d_{J'}(x) = \overline{d_{U_2}}(w) - 1$. It follows that $|J'| \le |J| \le n - 2m + 1 < |U_2|$ and that $|H'| = n - 1 - |J'| \le n - 1 - d_{J'}(x) \le 13n/18 - 1 < n - 2m + 2 = |U_1|$ since $5n/18 > 2m - 3$. Also note that $d_{H'}(x) = d_{T_n}(x) - d_{J'}(x) < 5n/18 \le \overline{d_{U_1}}(w) - 1$ since $d_{T_n}(x) \le \Delta(T_n) < 5n/9$ by Claim 4.4. Now applying the embedding described above with the sets $H'$ and $J'$ in place of $H$ and $J$ yields that $UC_n$ is a subgraph of $\overline{G}$ and a contradiction.
\end{proof}

\begin{claim}
There are subsets $U'_1 \subseteq U_1$ and $U'_2 \subseteq U_2$ such that $|U'_1| = |U_1| - 1$ and $|U'_2| = |U_2| - 1$ satisfying that for each $w \in W$, either $d_{U'_1}(w) \le m - 1$ or $d_{U'_2}(w) \le m - 1$.
\end{claim}

\begin{proof}
First we consider the case in which there are two disjoint edges in $\overline{G}$ between $U_1$ and $U_2$. By Claim 4.7, for each $w \in W$, either $\overline{d_{U_1}}(w) < 1 + 5n/18$ or $\overline{d_{U_2}}(w) < 1 + 5n/18$. Assume without loss of generality that $\overline{d_{U_1}}(w) < 1 + 5n/18$. If $S = N_{U_1}(w)$, then $|S| = d_{U_1}(w) > |U_1| - 5n/18 - 1$. Now assume for contradiction that $d_{U_2}(w) \ge m$ and let $x_1, x_2, \dots, x_m \in U_2$ be $m$ neighbors of $w$ in $U_2$. By Claim 4.6, $\overline{d_{U_1}}(x_i) \le 2m - 2$ for each $1 \le i \le m$ and thus
$$d_S(x_i) \ge |S| - \overline{d_{U_1}}(x_i) > |U_1| - 2m + 1 - 5n/18 \ge 13n/18 - 4m + 3 \ge m$$
for all $1 \le i \le m$ because $n \ge m^2 - m + 1 \ge 18(5m - 3)/13$ because $m \ge 8$. Thus there is a matching in $G$ between $x_1, x_2, \dots, x_m$ and a subset of size $m$ of $S$. These $2m$ vertices together with $w$ form an $F_m$ subgraph of $G$, which is a contradiction. Therefore $d_{U_2}(w) \le m - 1$. By a symmetric argument, if $\overline{d_{U_2}}(w) < 1 + 5n/18$ it follows that $d_{U_1}(w) \le m - 1$. Now taking any size $|U_1| - 1$ and $|U_2| - 1$ subsets of $U_1$ and $U_2$ as $U_1'$ and $U_2'$, respectively, yields the claim.

Now consider the case in which there are no two disjoint edges in $\overline{G}$ between $U_1$ and $U_2$. It follows that all of the edges between $U_1$ and $U_2$ in $\overline{G}$ must be incident to a common vertex $y$. Without loss of generality, assume that $y \in U_1$ and let $U_1' = U_1 - \{ y \}$. It follows that the parts $U_1'$ and $U_2$ induce a complete bipartite subgraph of $G$. If $d_{U_1'}(w) \ge m$ and $d_{U_2}(w) \ge m$ for any $w \in W$, this would again yield an $F_m$ subgraph of $G$ centered at $w$. Therefore either $d_{U_1'}(w) \le m - 1$ or $d_{U_2}(w) \le m - 1$. Similarly if $y \in U_2$, then setting $U_2' = U_2 - \{ y \}$ yields analogous results. Combining these cases yields the claim.
\end{proof}

The remainder of the proof of Theorem 5 is similar to the proof of Theorem 4 after Claim 3.6. Where details are the same, we refer to the proof of Theorem 4. Let $Z = V(G) - (X_1 \cup Y_1 \cup X_2 \cup Y_2)$ and note that $Z \subseteq W$ and $|Z| = 2m - 3$. As in the proof of Theorem 4, we may assume that there is a set $Z_1 \subseteq Z$ such that $|Z_1| = m - 1$ and each $w \in Z_1$ satisfies that $d_{U_1'} (w) \le m - 1$. Let $Z_1 = \{ z_1, z_2, \dots, z_{m-1} \}$. We now consider two cases.

\begin{casenew}
$|L(T_n)| \ge 2m + 2$.
\end{casenew}

Using a similar procedure as in Case 1 of the proof of Theorem 4, choose $D$ to be a set of $|D| = m - 1$ leaves chosen such that $t_2 \not \in D$, $d \le m - 4$ where $d$ is the maximum value $d_{D}(x)$ over all $x \in V(T_n) - D$. This is possible because $|L(T_n)| \ge 2m + 2$, Claim 4.4 implies that no $x \in V(T_n)$ is adjacent to more than $2m - 2$ leaves and $m \ge 3$. Define $d_i$ for $1 \le i \le m - 1$ and $y_i$ for $1 \le i \le k = |N_{T_n}(D)|$ as in the proof of Theorem 4. Since each leaf in $D$ is adjacent to one vertex $y_i$ and $d_D(y_j) = d$ for some $j$, it follows that $k \le m - d$. Consider the following embedding of $UC_n$ to $\overline{G}$. Begin by mapping $d_i$ to $z_i$ for all $1 \le i \le m - 1$. For each $1 \le i \le k$, define $N_i$ as in the proof of Theorem 4. By Claim 4.8, each $z_i$ is not adjacent to at most $m - 1$ vertices of $X_1 \cap U_1' \subseteq U_1'$ in $\overline{G}$. The number of vertices in $X_1 \cap U_1'$ adjacent to all vertices in $N_i$ in $\overline{G}$ is at least
\begin{align*}
|X_1 \cap U_1'| - |N_i| \cdot (m - 1) &\ge |X_1| - 1 - d(m - 1) \\
&\ge(m - 3 - d)(m - 1) \ge m - d \ge k
\end{align*}
where the next to last inequality follows from the fact that $1 \le d \le m - 4$. Therefore greedily selecting $k$ vertices $u_1, u_2, \dots, u_k$ such that $u_i$ is adjacent to all vertices in $N_i$ in $\overline{G}$ and $u_i \in X_1$ for all $1 \le i \le k$ succeeds. As in the proof of Theorem 4, we have that
$$|A(T_n) \cap (\{ y_1, y_2, \dots, y_k \} \cup \{t_1, t_2\})| + |A(T_n) \cap D| \le 2 + |N_1| + |N_2| + \cdots + |N_k| = m + 1.$$
By Claim 3.2, since $|V(T_n - D)| = n - m + 1 = |X_1| + |Y_1|$, the forest $T_n - D$ can be embedded to $X_1 \cup Y_1$ in $\overline{G}$ such that $y_i$ is mapped to $u_i$ for all $1 \le i \le k$ and $t_1$ and $t_2$ are mapped to vertices in $X_1$. This yields a successful embedding of $UC_n$ to $\overline{G}$, which is a contradiction.

\begin{casenew}
$|L(T_n)| \le 2m + 1$.
\end{casenew}

By Lemma 1, if $F = \{ t_1, t_2 \}$, there is a subset $K \subseteq A(T_n)$ such that $F \cap K = \emptyset$, each $w \in K$ satisfies $d_{T_n}(w) = 2$ and $w$ is not adjacent to any leaves of $T_n$, no two vertices in $K$ have a common neighbor and
$$|K| \ge \frac{1}{4} \left( n - 8 |L(T_n)| + 8 \right).$$
Now note that
$$|K| + |L(T_n)| \ge \frac{1}{4} (n - 4 |L(T_n)| + 8) \ge \frac{1}{4} \left( n - 8m + 4 \right) > 2m + 5$$
because $n \ge m^2 - m + 1 > 16m + 16$ for $m \ge 18$. Therefore $|K| + |L(T_n)| \ge 2m + 6$ and $|K| \ge 5$ since $|L(T_n)| \le 2m + 1$. Now let $D$ be a set of size $m - 1$ consisting of $\min\{|K|, m - 2 \} \ge 5$ vertices in $K$ and its remaining vertices chosen from $L(T_n) - \{ t_2 \}$. Since $K$ does not contain $t_1$ or $t_2$ and $t_1$ is not a leaf of $T_n$, $D$ also does not contain $t_1$ or $t_2$. Furthermore if $d$ is the maximum value of $d_D(x)$ over all $x \in V(T_n) - D$, then $d \le |D| - |D \cap K| + 1 \le m - 5$ since no vertex is adjacent to two vertices in $K$. Note that $D$ contains at least one leaf.

Define $d_i, y_i$ and $N_i$ as in Case 1. Since each vertex in $D$ has degree at most two and $D$ contains at least one leaf, it follows that $k \le 2m - 3$. As in Case 1, the number of vertices in $X_1 \cap U_1'$ adjacent to all vertices in $N_i$ in $\overline{G}$ is at least
$$|X_1 \cap U_1'| - |N_i| \cdot (m - 1) \ge (m - 3 - d)(m - 1) > 2m - 3 \ge k.$$
Therefore, it again follows that greedily selecting $k$ vertices $u_1, u_2, \dots, u_k$ such that $u_i$ is adjacent to all vertices in $N_i$ in $\overline{G}$ and $u_i \in X_1$ for all $1 \le i \le k$ succeeds. Now note that since $K \subseteq A(T_n)$, it follows that the neighbors of vertices in $K \cap D$ are in $B(T_n)$. Thus $A(T_n) \cap \{ y_1, y_2, \dots, y_k \}$ includes only neighbors of leaves in $L(T_n) \cap D$. By the same reasoning as in Case 2 in the proof of Theorem 4,
$$|A(T_n) \cap (\{ y_1, y_2, \dots, y_k \} \cup \{ t_1, t_2 \})| + |A(T_n) \cap D| \le 2 + |D| = m + 1.$$
As in the previous case, by Claim 3.2, since $|V(T_n - D)| = n - m + 1 = |X_1| + |Y_1|$, the forest $T_n - D$ can be embedded to $X_1 \cup Y_1$ in $\overline{G}$ such that $y_i$ is mapped to $u_i$ for all $1 \le i \le k$ and $t_1$ and $t_2$ are mapped to vertices in $X_1$. This yields a successful embedding of $UC_n$ to $\overline{G}$, which is a contradiction. This completes the proof of Theorem 5.

\section{Conclusions and Future Work}

We remark that Claim 3.4 can be strengthened to show that $\overline{d_{U_2}}(w) \le m - 2$ for each $w \in U_1$ and $\overline{d_{U_1}}(w) \le m - 2$ for each $w \in U_2$ by adapting an argument in \cite{zhang2015ramsey}. However, this stronger bound was not necessary for our proof of Theorem 4 and thus omitted for simplicity. Claim 4.6 can be strengthened in a similar way. Theorem 2 and Theorem 4 together show Conjecture 1 is true except for the finitely many pairs $(n, m)$ satisfying that $6 \le m \le 8$ and $m^2 - m + 1 \le n < 3m^2 - 2m - 1$. In fact, examining our proof of Theorem 4 carefully yields that Conjecture 1 has been proven to hold unless $6 \le m \le 8$ and $n < \min \{ 8m, 3m^2 - 2m - 1 \}$. We also believe that a careful analysis of our method could extend our proof of Theorem 4 to the case in which $m = 8$. A direction for future work is to refine our proof of Theorem 4 to reduce the number of these pairs. Another direction for future work is to reduce the requirement $m \ge 18$ in Theorem 5, which we conjecture is true for smaller values of $m$.

Another possible direction for future work is to study the minimum threshold $k(n)$ such that if $G$ is any graph with at least $k(n)$ cycles and $n$ vertices, then $R(G, F_m) \ge 2n$ for all $n$. Our results show that $k(n) \ge 2$ for all $n$. Let $C_s^{t}$ denote the graph consisting of $t$ edge-disjoint copies of the cycle $C_s$ that share a single common vertex. Another direction for future work is to investigate $R(T_n, C_s^{t})$ by varying $t$, $n$ or both. The general case appears to be difficult but examining this value for $C_4^{t}$ with small or potentially all values of $t$ may be reasonable.

\section*{Acknowledgements}

This research was conducted at the University of Minnesota Duluth REU and was supported by NSF grant 1358695 and NSA grant H98230-13-1-0273. The author thanks Joe Gallian for suggesting the problem and Benjamin Gunby and Joe Gallian for helpful comments on the manuscript.


\end{document}